\documentclass[11pt]{article}
\usepackage[utf8]{inputenc}
\usepackage[margin=1.1in]{geometry} 
\usepackage{amsmath,amsthm,amssymb}
\usepackage{color}
\usepackage{mathrsfs}
\usepackage{enumerate}
\usepackage{graphicx}
\usepackage{caption}
\usepackage{mathtools}
\usepackage{dsfont}
\usepackage{todonotes}
\usepackage{comment}
\usepackage{titling}

\usepackage{hyperref}
\usepackage{amsrefs}

\newtheoremstyle{introdef}% hnamei
{3pt}% hSpace abovei
{3pt}% hSpace belowi
{\itshape}% hBody fonti
{}% hIndent amounti
{\mdseries}% hTheorem head fonti
{}% hPunctuation after theorem headi
{0pt}% hSpace after theorem headi
{}% hTheorem head spec (can be left empty, meaning ‘normal’)i

\newcommand{\abtor}{\mathrm{Ab}_{\mathrm{tor}}}
\newcommand{\BC}{\mathbb C} \newcommand{\BH}{\mathbb H}
\newcommand{\BR}{\mathbb R} 
\newcommand{\BN}{\mathbb N} 
 \newcommand{\BZ}{\mathbb Z}

\newcommand{\CG}{\mathcal G}

 \newcommand{\CN}{\mathcal N}
 \newcommand{\CP}{\mathcal P}

\DeclareMathOperator{\PSL}{PSL} % Spezielle lineare Gruppe
\DeclareMathOperator{\Isom}{Isom} % Isometrien einer Mf
\DeclareMathOperator{\Hom}{Hom} % Homomorphismen
\DeclareMathOperator{\vol}{vol} % Volumen

\DeclareMathOperator{\id}{id}

\DeclareMathOperator{\diam}{diam}

\DeclareMathOperator{\rank}{rank}

\DeclareMathOperator{\length}{length}

\DeclareMathOperator{\TCH}{TCH}
\DeclareMathOperator{\TCC}{TCC}

\DeclareMathOperator{\CC}{CC}
\DeclareMathOperator{\hull}{CH}
\DeclareMathOperator{\sys}{sys}

\makeatletter
\DeclareRobustCommand\bigop[2][1]{%
  \mathop{\vphantom{\sum}\mathpalette\bigop@{{#2}{#1}}}\slimits@
}
\newcommand{\bigop@}[2]{\bigop@@#1#2}
\newcommand{\bigop@@}[3]{%
  \vcenter{%
    \sbox\z@{$#1\sum$}%
    \hbox{\resizebox{#3\dimexpr\ifx#1\displaystyle.9\fi\dimexpr\ht\z@+\dp\z@}{!}{$\m@th#2$}}%
  }%
}
\makeatother

\newcommand{\bigast}{\DOTSB\bigop[1]{*}}

\usepackage{relsize}

\newtheorem{theorem}{Theorem}

\newtheorem{lemma}{Lemma}
\newtheorem{claim}[lemma]{Claim}
\newtheorem{example}[lemma]{Example}
\newtheorem{question}[lemma]{Question}
\newtheorem*{lemma*}{Lemma}
\newtheorem*{example*}{Example}
\newtheorem*{theorem*}{Theorem}
\newtheorem*{prop*}{Proposition}
\newtheorem{prop}[lemma]{Proposition}

\newtheorem*{cor*}{Corollary}
\newtheorem*{fact*}{Fact}

\newtheorem*{claim*}{Claim}

\theoremstyle{definition}
\newtheorem{defn}[lemma]{Definition}

\theoremstyle{introdef}
\newtheorem*{empty*}{}

\numberwithin{lemma}{section}

\DeclareMathOperator\arcosh{arcosh}

\newcommand{\Z}{\mathbb{Z}}
\newcommand{\C}{\mathbb{C}}

\newcommand{\CD}{\mathcal{D}}

\title{Combinations and chromatography of paths of Kleinian groups}
\author{Matthew Zevenbergen}
\date{}

\begin{document}

\maketitle

\vspace{-.6in}

\begin{abstract}
    We study continuous paths in the Chabauty topology on the set $\mathcal{D}_n$ of torsion free discrete subgroups of the isometry group of $n$-dimensional hyperbolic space. We prove a combination theorem for paths in $\mathcal{D}_n$, which allows us to construct an exotic path of discrete subgroups along which no two subgroups are isomorphic. We also introduce a technique we refer to as ``chromatography" to prove a decomposition theorem that characterizes paths of convex cocompact groups in $\mathcal{D}_3$.
\end{abstract}

\section{Introduction}

Let $\CD_n$ denote the set of all torsion free discrete subgroups of $\mathrm{Isom}^+(\BH^n)$, the group of orientation preserving isometries of hyperbolic $n$-space, and equip $\CD_n$ with the Chabauty topology (see Definition \ref{def:Chabauty}). The goal of this paper is to study continuous paths $P:[0,1]\longrightarrow\CD_n$, each of which describes a continuous deformation between the hyperbolic manifold quotients $P(0)\backslash\BH^n$ and $P(1)\backslash\BH^n$. We  prove a combination theorem to produce exotic examples of paths, and use a technique we call chromatography to prove a decomposition theorem that describes paths of convex cocompact groups in $\CD_3$.\par

Chabauty topology convergence of discrete subgroups of $\Isom^+(\BH^3)$, also called Kleinian groups, is a fundamental notion in the study of hyperbolic $3$-manifolds. Also known as the topology of \textit{geometric convergence}, the Chabauty topology appears in Thurston's hyperbolization theorem \cite{THUI}, Thurston's hyperbolic Dehn surgery theorem (see \cite{THUR} and \cite{BEN}), and in numerous other significant results (e.g. \cite{LimitTameness}, \cite{BCM-EndingLaminations}, \cite{DICKCOVER}, \cite{JORMAR}, \cite{NAM}). The term ``geometric convergence" is justified by the equivalence between Chabauty topology convergence of Kleinian groups and framed Gromov-Hausdorff convergence of the quotient hyperbolic manifolds \cite[Ch. E]{BEN}.

\vspace{-.1in}

\paragraph{Combinations and examples.}

Classical theory pioneered by Ahlfors, Bers, Marden, and Sullivan (see \cite{MAT}) studies deformations of hyperbolic structures on a $3$-manifold $M$ from the perspective of discrete embeddings of the \textit{fixed} group $\pi_1M$ into $\Isom^+(\BH^3)$. The space $\CD_3$, however, contains \textit{all torsion free discrete subgroups}, with no restriction on isomorphism type. The present work considers deformations between a pair of hyperbolic manifolds $P(0)\backslash\BH^3$ and $P(1)\backslash\BH^3$ from the perspective of a path $P:[0,1]\longrightarrow\CD_3$. Our first aim is to differentiate these notions of deformations by investigating how the isomorphism type of the groups $P(t)$ can change along such a path.\par

There is a rich history of examples of convergent \textit{sequences} $\Gamma_i\rightarrow\Gamma$ in $\CD_3$ where $\Gamma$ is not isomorphic to any of the approximating groups $\Gamma_i$ (e.g. \cite{Brock-Iteration}, \cite{Kerckhoff-Thurston}, \cite{NamaziSouto-Heegaard}). Many of these constructions, however, do not extend to paths. For example, Kerckhoff and Thurston \cite{Kerckhoff-Thurston} construct a sequence of closed genus-$2$ surface groups in $\CD_3$ that converge to a group containing a subgroup isomorphic to $\BZ^2$. It follows from Corollary 4.7 of \cite{ZevenbergenPaths1}, however, that if $P:[0,1]\longrightarrow\CD_3$ is continuous and $\BZ^2\leq P(1)$, then $\BZ^2\leq P(t)$ for $t$ close to $1$, so the Kerckhoff-Thurston sequence cannot extend to a path with the same limit.\par

In contrast with the classical study of discrete representations of a fixed group, one can, however, construct paths in each $\CD_n$ along which the isomorphism type is non-constant. Simple examples in $\CD_2$ arise by continuously pinching a simple closed geodesic on the quotient closed hyperbolic surfaces, leaving a cusped surface in the limit, as described in \cite{IANCHAB}. More examples in $\CD_3$ are discussed by the author in \cite{ZevenbergenPaths1}.\par

In this paper, we construct a path $\Phi:[0,1]\longrightarrow \CD_3$ for which the isomorphism type of the groups $\Phi(t)$ fails to be constant as dramatically as possible.

\begin{theorem}[\hypertarget{thm:binary}{Binary Path}]
\label{thm:binary path}
    There exists a path $\Phi:[0,1]\longrightarrow\CD_3$ such that $\Phi(t)$ is isomorphic to $\Phi(s)$ if and only if $t=s$.
\end{theorem}

We call $\Phi$ the ``Binary Path" because it will be constructed so that a binary representation for $t$ can be reconstructed from the torsion of the abelianization of the group $\Phi(t)$, for all $t\in[0,1]$. Here and in what follows, a \textit{path} $P:I\longrightarrow\CD_n$ is always assumed to be continuous, with $I$ an interval. For each $n\geq 3$, one can include $\CD_3\subset\CD_n$ by identifying $\BH^3$ with a $3$-dimensional hyperplane in $\BH^n$, so the theorem also holds for such $\CD_n$. We note that such a path cannot exist in $\CD_2$, however, since there are only countably many isomorphism types of groups in $\CD_2$.\par

The technical tool we use to justify the continuity of the binary path is a combination theorem inspired by those of Klein and Maskit. The Klein-Maskit combination theorems allow one to combine known discrete subgroups of $\Isom^+(\mathbb{H}^n)$, producing a discrete amalgamated free product or HNN-extension of the constituent groups (see \cite{MASK}). These theorems have played a key role in the study of Kleinian groups and hyperbolic manifolds, most notably in the inductive step of Thurston's hyperbolization for Haken manifolds (\cite{THUI}), and many more applications (e.g. \cite{ANCAN}, \cite{ZevenbergenHuang}, \cite{MaskitBoundaries}).\par

The following combination theorem allows one to combine collections of paths in $\CD_n$ to produce new paths.

\begin{theorem}[\hypertarget{thm:combo}{Combination Theorem}]
\label{thm:combo}
    Let $\CP$ be a collection of paths $P:I\longrightarrow\CD_n$ defined on a fixed interval $I$, and for each $P\in \CP$ assume that $A_P\subset\BH^n$ is a non-empty closed subset. If
    \begin{enumerate}
        \item $\exists \;\delta>0$ such that for all $P\neq P'$, $d(A_P,A_{P'})>\delta$ and the closures of $A_P$ and $A_{P'}$ in $\overline{\BH^n}$ are disjoint,
        \item for all $D>0$ and $P\in \CP$, there are finitely many $P'\in \CP$ such that $d(A_P,A_{P'})<D$, and
        \item for all $s\in I$, $P\in \CP$, and $\psi_s\in P(s) -\{\id\}$, we have $\psi_s(\BH^n-A_P)\subset A_P,$
    \end{enumerate}
    then, the group \[\Pi(s):=\left<\bigcup_{P\in\CP}P(s)\right>\] is discrete for all $s\in I$, and $\Pi:I\longrightarrow\CD_n$ is a path in the Chabauty topology on $\CD_n$.
\end{theorem}

Here, $\overline{\BH^n}:=\BH^n\cup S^{n-1}$ is the union of $\BH^n$ with the sphere at infinity. Note that conditions (1) and (3) here are nearly identical to the hypotheses of the classical ping-pong combination theorem (see \cite{MAT} and Subsection \ref{subsec:combo thm}). Indeed, if all of the paths in $\CP$ are constant, then the conclusion of the theorem simply says that $\Pi(s)$ is discrete for all $s$, and condition (2) is unnecessary. The main task for proving the present theorem is therefore to justify \textit{continuity} of $\Pi$. Example \ref{ex:non} will describe a family of paths $\CP$ for which $\Pi(s)$ is discrete for every $s$, but condition (2) of the \hyperlink{thm:combo}{Combination Theorem} is not satisfied, and the map $\Pi$ fails to be continuous.

\vspace{-.1in}

\paragraph{Decompositions and chromatography.} 

Paths in $\CD_n$ that are produced by the \hyperlink{thm:combo}{Combination Theorem} have a natural decomposition into the set of paths $\CP$ that appears in the construction. We consider next the extent to which general paths in $\CD_n$ can be decomposed into simpler pieces. The following definition makes our notion of ``decomposition" precise.

\begin{defn}[Free decomposition]
\label{def:free decomp}
    If $\Pi:I\longrightarrow\CD_n$ is a path for which there exists a collection $\CP$ of paths $P:I\longrightarrow\CD_n$ such that \[\Pi(s)=\bigast_{P\in\CP}P(s)\] for all $s\in I$, we say $\Pi$ \textit{freely decomposes into $\CP$,} and call $\CP$ a \textit{free decomposition} of $\Pi$.
\end{defn}

In an abuse of notation, for a collection $\CG$ of subgroups $G\leq\Isom^+(\BH^n)$ and a group $\Gamma\leq\Isom^+(\BH^n)$, we say \[\Gamma=\bigast_{G\in\CG}G\] if $\Gamma$ contains every subgroup $G\in\CG$ and the natural homomorphism from the free product $\bigast_{G\in\CG}G$ to $\Gamma$ is an isomorphism. In this case, we say that each $G\in\CG$ is a \textit{free factor} of $\Gamma$. \par 

We remark that, indeed, a path $\Pi$ produced using the \hyperlink{thm:combo}{Combination Theorem} freely decomposes into the collection of paths $\CP$ from which it is generated.\par

Recall that the deformations of discrete subgroups of $\Isom^+(\BH^3)$ classically studied by Ahlfors, Bers, Marden, Sullivan, and many others (see \cite{MAT}) are perturbations within a fixed isomorphism class, from the perspective of representations. Accordingly, it is natural to use such deformations as the basis for what will be the atomic pieces of our decompositions. 

\begin{defn}[Isomorphism bump path]
\label{def:Isom bump path}
    A path $B:I\longrightarrow\CD_n$ is called an \textit{isomorphism bump path} if for some discrete group $\Gamma$ and subinterval $J\subset I$, there exists a continuous map $R:J\longrightarrow\Hom(\Gamma,\Isom^+(\BH^n))$ such that $R(t)$ is an injective homomorphism for all $t$, and \[B(t)=
    \begin{cases}
        (R(t))(\Gamma), &t\in J \\
        \{\id\}, &t\in I-J.
    \end{cases}\]
\end{defn}

Here, $\Hom(\Gamma,\Isom^+(\BH^n))$ is equipped with the compact-open topology, which is equivalent to that of pointwise convergence since $\Gamma$ is discrete. Informally, Proposition \ref{prop:J Facts} and Lemmas \ref{lem:injhomo} and \ref{lem:bump vs J} combine to say that if $J\subset\mathrm{int}(I)$, one should envision an isomorphism bump path as tracking perturbations of the group $\Gamma$ coming in from infinity at $\inf J$, and then going back out to infinity at $\sup J$.\par

We now state our main decomposition theorem for paths of convex cocompact groups in $\CD_3$. A discrete group $\Gamma\leq\Isom^+(\BH^n)$ is called convex cocompact if there exists a non-empty closed convex $\Gamma$-invariant subset $C\subset\BH^n$ such that $\Gamma\backslash C$ is compact.

\begin{theorem}[\hypertarget{thm:decomp}{Chromatography}]
\label{thm: decomp}
    If $P:[0,1]\longrightarrow\CD_3$ is a path such that $P(t)$ is convex cocompact for all $t$, then there exists a path $D:[0,1]\longrightarrow\CD_3$ such that \[D(0)=P(0),\hspace{.5in} D(t)\text{ is a free factor of } P(t)\;\;\forall t,\] and $D$ freely decomposes into a finite set of isomorphism bump paths.
\end{theorem}

We refer to Theorem \ref{thm: decomp} as ``chromatography," a technique from chemistry used to separate the components of a chemical mixture. In paper chromatography, a solvent pulls a mixture of chemicals along a segment of paper, and different components of the mixture are deposited at different points along the paper. Here, one should informally think of the path $D$ as pulling the group $P(0)$ through the groups $P(t)$, and the free decomposition of $D$ as separating $P(0)$ into free factors that diverge to infinity at different points as they are pulled along the interval. More precisely, the path $D$ freely decomposes into isomorphism bump paths $B_1,...,B_k:[0,1]\longrightarrow\CD_3$, for each of which there is some $t_i\in(0,1]$ such that $B_i(t)\neq\{\id\}$ if and only if $t<t_i$. The group $P(0)$ is then a freely generated combination, or ``mixture," of the factors $B_i(0)$, and the points $t_i$ (where $B_i(t)\rightarrow\{\id\}$ in the Chabauty topology) are where different factors of this mixture diverge to infinity, analogous to the points on a piece of chromatography paper where components of a chemical mixture are left behind. A cartoon of this is shown in Figure \ref{fig:chromatography}.\par

\begin{figure}[h]
    \centering
    \includegraphics[width=1\linewidth]{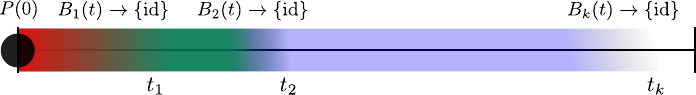}
    \caption{A cartoon of the Chromatography Theorem. At each time $t_i$, one of the isomorphism bump paths $B_i$ in the free decomposition for $D$ converges to the trivial group. }
    \label{fig:chromatography}
\end{figure}

We remark that the hypothesis in the \hyperlink{thm:decomp}{Chromatography Theorem} that each group $P(t)$ is convex cocompact is essential. Indeed, the example in $\CD_2$ of pinching a simple closed curve on a surface produces a path of closed surface groups that converges to the fundamental group of a \textit{cusped} surface, which is free. Closed surface groups, however, do not admit a non-trivial free splitting, so the theorem does not hold for arbitrary paths of geometrically finite groups (see also \cite[Thm. C]{ZevenbergenPaths1}). In $3$-dimensions degenerate ends introduce similar challenges to those posed by cusps. \par

One may hope that it is possible to improve the \hyperlink{thm:decomp}{Chromatography Theorem} by showing that every path of convex cocompact groups $P:I\longrightarrow\CD_3$ freely decomposes into a collection of isomorphism bump paths. Indeed, the \hyperlink{thm:binary}{Binary Path} and all other paths that our \hyperlink{thm:combo}{Combination Theorem} produces from a collection of isomorphism bump paths will have such a decomposition. The following example will show, however, that this is not the case.

\vspace{.1in}

    \noindent\textbf{\hyperref[ex:decomp non-ex]{Example~\ref*{ex:decomp non-ex}.}}
\textit{There exists a convex cocompact path $P:[0,1]\longrightarrow\CD_2$ that does not freely decompose into a collection of isomorphism bump paths.}

\vspace{.1in}

\noindent Here, a path $P:I\longrightarrow\CD_n$ is called \textit{convex cocompact} if $P(t)$ is convex cocompact for all $t\in I$. In brief, Example \ref{ex:decomp non-ex} is constructed from a rank-$3$ free group by sending a generator from one free basis to infinity as $t\rightarrow 0$, and then sending a generator from a \textit{different} free basis to infinity as $t\rightarrow 1$.

\vspace{-.1in}

\paragraph{Organization.} Section \ref{sec:combos} focuses on the \hyperlink{thm:combo}{Combination Theorem}, with Subsection \ref{subsec:combo thm} dedicated to the proof, and in Subsection \ref{subsec:non-ex}, we produce a non-example that demonstrates the necessity of some hypotheses. The \hyperlink{thm:binary}{Binary Path} is constructed in Section \ref{sec:binary path}, with Subsection \ref{subsec:specifying homology} containing a topological lemma, and Subsection \ref{subsec:binary construction} completing the construction. In Section \ref{sec:decomp} we turn our attention towards decompositions of paths. Subsection \ref{subsec:path machinery} reviews and expands upon machinery for studying paths in $\CD_n$ that was developed by the author in \cite{ZevenbergenPaths1}. The proof of the \hyperlink{thm:decomp}{Chromatography Theorem} is carried out in Subsection \ref{subsec:decomp thm}, which starts with a number of technical lemmas needed to justify the free splittings that appear in the theorem. We end with Subsection \ref{subsec:decomp counterex}, which contains Example \ref{ex:decomp non-ex}, showing that not every path of convex cocompact groups freely decomposes into isomorphism bump paths.

\vspace{-.1in}

\paragraph{Acknowledgments.} The author thanks Ian Biringer for introducing him to the topic, for many helpful discussions, and for feedback on a draft of this paper. The author would also like to thank the organizers and participants of the Topology Students Workshop 2026 for conversations that greatly improved the introduction section of this paper. Finally, the author thanks Eva Bayer for sharing her chromatography expertise.

\section{Combinations of paths}
\label{sec:combos}

\subsection{The combination theorem}
\label{subsec:combo thm}

We start by recalling some details from classical combination theorems. Assume $\CG$ is a collection of torsion free discrete subgroups of $\Isom^+(\BH^n)$, for some $n\geq2$. If $\{A_\Gamma\}_{\Gamma\in\CG}$ is a collection of pairwise disjoint closed subsets of $\BH^n$ such that $\BH^n-\bigcup_{\Gamma\in\CG} A_\Gamma$ is a non-empty open set and \begin{equation}
\label{eqn:discreteness condition}
    \psi(\BH^n-A_\Gamma)\subset A_\Gamma,
\end{equation} for each $\Gamma\in\CG$ and $\psi\in\Gamma-\{\id\}$, 
then it follows from standard arguments in \cite{MacbeathFreeProd} and \cite{MASK} that the group $\Lambda$ generated by the union of the subgroups in $\CG$ is discrete and splits as a free product: \[\Lambda\cong\bigast_{\Gamma\in\CG}\Gamma.\] \par

Identifying $\Lambda$ with this free product, assume that $\psi:=\gamma_k\cdot...\cdot \gamma_1$ is a reduced word, with each $\gamma_i\in\Gamma_i\in\CG$, and $k\geq 2$. For each $i$, define \[w_i:=\gamma_k\cdot...\cdot\gamma_i\hspace{.2in} \text{ and }\hspace{.2in} A_i:=A_{\Gamma_{i}}.\] A classical ping-pong argument shows that for every $q\in\BH^n-\bigcup_{\Gamma\in \CG} A_\Gamma$, we have 

\begin{align}
\label{eqn:nested inclusions}
    \psi(q)\in w_2(A_1)\subset w_2(\BH^n-A_2)\subset...\subset w_i&(A_{i-1})\subset w_i(\BH^n-A_i)\subset\\&...\subset\gamma_k(A_{k-1})\subset \gamma_k(\BH^n- A_k)\subset A_k\nonumber
\end{align}

We now use this sequence of nested inclusions to obtain the following inequality.

\begin{claim}
\label{claim:nested distances}
    \[d(q,\gamma_k(A_{k-1}))+d(\gamma_1(q),A_2)+\sum_{1<i<k}d(\gamma_{i}(A_{i-1}),A_{i+1})\leq 2 d(q,\psi(q))\]
\end{claim}

\begin{proof}
    Let $\alpha$ be the geodesic segment between $q$ and $\psi(q)$. By Equation \ref{eqn:nested inclusions}, there is a subsegment $\alpha_k\subset\alpha$ with endpoints on $q$ and $\partial \gamma_k(A_{k-1})$, which has length at least $d(q,\gamma_k(A_{k-1}))$. Similarly, there is a subsegment $\alpha_1\subset\alpha$ between $\psi(q)$ and $\partial w_2(\BH^n-A_2)$, with \[\length(\alpha_1)\geq d(\psi(q),\partial w_2(\BH^n-A_2))=d(\gamma_1(q),\partial(\BH^n-A_2))=d(\gamma_1(q),A_2),\] where the last equality follows from noting that $\gamma_1(q)\in A_1\subset\BH^n-A_2$.\par

    For $1<i<k$, there is a subsegment $\alpha_i\subset\alpha$ with endpoints on $\partial w_{i+1}(\BH^n-A_{i+1})$ and $w_{i}(A_{i-1})$. Similar to above, we then have that for $1<i<k$, \[\length(\alpha_i)\geq d(\gamma_i(A_{i-1}),\partial(\BH^n-A_{i+1}))=d(\gamma_i(A_{i-1}),A_{i+1}).\] The claim now follows by noting that $\alpha_i\cap\alpha_j=\varnothing$ if $|i-j|\geq2$.
\end{proof}

Application of Claim \ref{claim:nested distances} will be aided by the following lemma.

\begin{lemma}
\label{lemma:bounded subsets}
    Suppose $A,B,C\subset\BH^n$ are nonempty closed subsets whose closures in $\overline{\BH^n}$ satisfy $\overline{B}\cap(\overline{A}\cup\overline{C})=\varnothing$. Then, for all $D>0$, the set of elements $\psi\in\Isom^+(\BH^n)$ such that \[d(C,\psi(A))\leq D\hspace{.25in}\text{and}\hspace{.25in}\psi^i(A)\subset B\] for all $i\neq0$ is bounded in $\Isom^+(\BH^n)$.
\end{lemma}

\begin{proof}

    Assume that $\{\psi_j\}_{j\in\BN}\subset\Isom^+(\BH^n)$ is a sequence of isometries such that $\psi_j^i(A)\subset B$ for all $j$ and all $i\neq 0$. Since $A\cap B=\varnothing$, each $\psi_j$ must be parabolic or hyperbolic. Up to passing to a subsequence, we can assume that the attracting and repelling fixed points of $\psi_j$ (which are the same if $\psi_j$ is parabolic) converge to points $p,q\in\partial\overline{\BH^n}$, respectively, as $j\rightarrow\infty$. Since $\psi_j^i(A)\subset B$ for all $j$ and $i\neq 0$, we must have $p,q\in \overline{B}$.\par

    Suppose that the sequence $\{\psi_j\}$ is unbounded. It follows that if $K\subset\BH^n$ is any compact subset, we have 
    \begin{equation}
    \label{eqn:empty intersection}
        \psi_j(A)\cap K=\varnothing
    \end{equation} for sufficiently large $j$, since $q\not\in \overline{A}$. Defining \[K:=\{x\in B\;|\;d(x,C)\leq D\},\] observe that $K$ is compact, since $\overline{B}\cap\overline{A}=\varnothing$, so the conclusion follows from Equation \ref{eqn:empty intersection}.
\end{proof}

We now proceed with the proof of the Combination Theorem, Theorem \ref{thm:combo}. The primary goal of the proof will be to justify the Chabauty topology continuity of the path $\Pi:I\longrightarrow\CD_n$ that the theorem produces. The following definition characterizes convergence in the Chabauty topology on $\CD_n$.

\begin{defn}[Chabauty topology]
\label{def:Chabauty}
    A sequence $\{\Gamma_i\}\subset\CD_n$ converges to $\Gamma$ in the\textit{ Chabauty topology} on $\CD_n$ if the following hold:
\vspace{-.1in}
    \begin{enumerate}[{Condition} 1:]
        \item If a sequence of elements $\psi_i\in\Gamma_i$ has a subsequence converging to $\psi\in\Isom^+(\BH^n)$, then $\psi\in\Gamma$.
        \vspace{-.1in}
        \item For each $\psi\in\Gamma$, there exists a sequence of elements $\psi_i\in\Gamma_i$ with $\psi_i\rightarrow\psi$ in $\Isom^+(\BH^n)$.
    \end{enumerate}
\end{defn}

\begin{proof}[Proof of the \hyperlink{thm:combo}{Combination Theorem}]

That the group $\Pi(s)$ is discrete for all $s\in I$ follows immediately from classical combination results (see \cite{MAT}), as appears in the discussion preceding Claim \ref{claim:nested distances}. Indeed, condition (1) implies that the complement of $\bigcup_{P\in\CP}A_P$ is non-empty and open, and condition (3) exactly matches with Equation \ref{eqn:discreteness condition}. Recall that it also follows that each group $\Pi(s)$ is isomorphic to the free product $\bigast_{P\in\CP}P(s)$.\par

Fixing $s\in I$, it remains to check that $\Pi(t)\rightarrow\Pi(s)$ in the Chabauty topology on $\CD_n$ as $t\rightarrow s$. That the second condition for convergence in $\CD_n$ as in Definition \ref{def:Chabauty} is satisfied follows immediately from the continuity of each of the constituent paths $P\in\CP$, from which it follows that each $\varphi\in P(s)$ is the limit of elements of each $P(t)$ as $t\rightarrow s$. \par

To address the first condition for convergence in the Chabauty topology, suppose that $t_i\rightarrow s$ as $i\rightarrow\infty$ and that a sequence of elements $\psi_i\in\Pi(t_i)$ converges to $\psi\in\Isom^+(\BH^n)$. Our goal is to show that $\psi\in\Pi(s)$. We consider each $\psi_i$ as the reduced word\[\psi_i=\gamma_{i,k_i}\cdot...\cdot\gamma_{i,1}\] in $\bigast_{P\in\CP}P(t_i)$ under the natural identification of $\Pi(t_i)$ with this free product, with each $\gamma_{i,j}\in P_{i,j}(t_i)$. Let $A_{i,j}$ denote $A_{P_{i,j}}$ and fix an element $q\in\BH^n-\bigcup_{P\in\CP}A_P$. \par

We first suppose that there are infinitely many $i$ such that $k_i=1$, in which case we may pass to a subsequence so that each $\psi_i\in P_i(t_i)$ for some $P_i\in\CP$. Since each $\psi_i(q)$ is contained in the closed disjoint union $\bigcup_{P\in\CP}A_P$, there must be some $P_*\in\CP$ such that $\psi(q)\in A_{P_*}$. It then follows that $\psi_i(q)\in A_{P_*}$ for $i$ sufficiently large, which implies that $P_i=P_*$ for such $i$. Since $P_*(t_i)\rightarrow P_*(s)$ in the Chabauty topology, the first condition in Definition \ref{def:Chabauty} implies that $\psi\in P_*(s)$, as desired.\par 

We now assume that $k_i\geq 2$ for all but finitely many $i$, in which case we may pass to a subsequence so that $k_i\geq 2$ for all $i$, which will allow us to apply Claim \ref{claim:nested distances}. Define \[D:=2\cdot d(q,\psi(q))+1.\] Claim \ref{claim:nested distances} implies that for $i$ sufficiently large, we have 
\begin{equation}
    \label{eqn:nested distances proof version}
    d(q,\gamma_{i,k_i}(A_{i,k_i-1}))+d(\gamma_{i,1}(q),A_{i,2})+\sum_{1<j<k_i}d(\gamma_{i,j}(A_{i,j-1}),A_{i,j+1})< D.
\end{equation}
Since $d(A_P,A_{P'})>\delta$ for all $P\neq P'$ by condition (1), all terms on the left-hand side of Equation \ref{eqn:nested distances proof version}, except possibly the first, are greater than $\delta$. Hence, we have $(k_i-1)\delta<D$, so we can conclude that $k_i\leq1+D/\delta$ for $i$ sufficiently large. Thus, passing to a subsequence allows us to assume that there is some fixed $k$ such that $k_i=k$ for all $i$, and that Equation \ref{eqn:nested distances proof version} holds for all $i$. \par

Since $\bigcup_{P\in\CP}A_P$ is a closed union of disjoint sets, and $\psi(q)$ is the limit of the points $\psi_i(q)$ that are contained in this union, there exists some element $P_*\in\CP$ such that $\psi(q)\in A_{P_*}$. Further, condition (1) implies that we must have $\psi_i(q)\in A_{P_*}$, and hence $P_{i,k}=P_*$, for $i$ sufficiently large. Since each term on the left-hand side of Equation \ref{eqn:nested distances proof version} must be less than $D$, it follows that we have $d(A_{i,j},A_{i,j+1})<D$ for all $j<k$. Given that $A_{i,k}=A_{P_*}$, condition (2) implies that there are only finitely many choices for each $P_{i,{k-1}}$, independent of $i$. Continuing inductively, there is similarly a finite set of possibilities for $P_{i,j}$ for each $j$. Hence, passing to a further subsequence allows us to assume that for each $j$, there is a fixed $P_j\in\CP$ such that $P_{i,j}=P_j$ for all $i$.\par

We now simplify notation so that $\psi$ is a limit of reduced words \[\psi_i=\gamma_{i,k}\cdot...\cdot\gamma_{i,1}\in\Pi(t_i), \hspace{.2in}\gamma_{i,j}\in P_j(t_i)\;\;\forall j,\hspace{.2in}\text{with}\hspace{.2in} A_j:=A_{P_j}.\] We claim that the sequence $\{\gamma_{i,j}\}_{i}$ is bounded for all $j$. Indeed, for all $j$ with $1<j<k$, Equation \ref{eqn:nested distances proof version} implies that $d(\gamma_{i,j}(A_{j-1}),A_{j+1})<D$ for all $i$, so it follows from condition (1) and Lemma \ref{lemma:bounded subsets} that the sequences $\{\gamma_{i,j}\}_i$ are bounded. We also have from Equation \ref{eqn:nested distances proof version} that $d(q,\gamma_{i,k}(A_{k-1}))<D$ and $d(\gamma_{i,1}(q),A_2)<D$ for all $i$, from which we also conclude from Lemma \ref{lemma:bounded subsets} that the sequences $\{\gamma_{i,k}\}_i$ and $\{\gamma_{i,1}\}_i$ are bounded.\par

By passing to a subsequence, we can then conclude that for each $j$, there exists $\gamma_j\in\Isom^+(\BH^n)$ such that $\gamma_{i,j}\rightarrow\gamma_j$ as $i\rightarrow\infty$. We have $\gamma_{i,j}\in P_j(t_i)$ for all $i,j$, so it follows from Chabauty topology continuity of each path $P_j$ that $\gamma_j\in P_j(s)$. Hence, we conclude that \[\psi_i=\gamma_{i,k}\cdot...\cdot\gamma_{i,1}\longrightarrow \gamma_k\cdot...\cdot\gamma_1\in\Pi(s)\] as $i\rightarrow\infty$, as was to be shown.
\end{proof}

\subsection{A non-example with weakened hypotheses}
\label{subsec:non-ex}

The goal for this subsection is to construct Example \ref{ex:non}, which demonstrates the necessity for condition (2) of the \hyperlink{thm:combo}{Combination Theorem}. We start with Example \ref{ex:trivial infinitely often}. The construction in the proof of Example \ref{ex:trivial infinitely often} that we provide is more complicated than is strictly necessary to prove the example statement, but the technique described will be used in the later construction of Example \ref{ex:non}.

\begin{example}[Isomorphism type changing infinitely many times, Convex cocompact]
\label{ex:trivial infinitely often}
    There exists a path $P:[0,1]\longrightarrow\CD_2$ such that $P(t)$ is non-trivial for almost every $t\in[0,1]$ and is convex cocompact for all $t$, but there are infinitely many $t\in[0,1]$ such that $P(t)=\{\id\}$.
\end{example}

\begin{proof}

Let $\{\psi_i\}_{i\geq 1}$ be a sequence of hyperbolic isometries in $\Isom^+(\BH^2)$ such that each $\psi_i$ has translation length $L_i$ with $L_i\rightarrow\infty$ as $i\rightarrow\infty$. For $t>0$, define $\psi_i^t$ to be the hyperbolic isometry with translation length $L_i/t$ and same attracting and repelling fixed points as $\psi_i$. In particular, $\psi_i^1=\psi_i$.\par

We first define a map $Q:[1,\infty)\longrightarrow\CD_2$ by 
    \[Q(t):=
    \begin{cases}
        \langle \psi_i^{t-2i+1}\rangle, &2i-1<t\leq2i \text{ for some }i\geq 1\\
        \langle \psi_i^{2i+1-t}\rangle, &2i\leq t <2i+1\text{ for some }i\geq 1\\
        \{\id\}, &\text{else}
    \end{cases}\]
Note that $Q(2i)=\langle \psi_i\rangle$ for all $i\geq 1$. Informally, on each interval $(2i-1,2i]$, $Q(t)$ is generated by a hyperbolic isometry whose translation length is decreasing from infinity, and converging to $L_i$. On each $[2i,2i+1)$, $Q(t)$ is generated by a hyperbolic isometry whose translation length increases from $L_i$ and diverges to $\infty$. \par

Since the isometries $\psi_i^t$ vary continuously in $\Isom^+(\BH^2)$ as $t>0$ varies, $Q(t)$ is continuous on each interval $(2i-1,2i+1)$. Further, as $t\rightarrow 2i\pm1$ for any $i$, the translation length of the generators of the cyclic groups $Q(t)$ diverge to $\infty$, so $Q(t)\rightarrow\{\id\}$. Thus, $Q$ is a path in $\CD_2$.\par

We now define $P:[0,1]\longrightarrow\CD_2$ by \[P(t):=
\begin{cases}
    Q(1/t), &0<t<1\\
    \{\id\},&\text{else}.
\end{cases}\]
Since it will be important for Example \ref{ex:non}, we note that 
\begin{equation}
    \label{eqn:equality at times}
    P(1/2i)=\langle\psi_i\rangle, \hspace{.2in}\forall\; i\geq 1
\end{equation}

We now claim that $P$ is continuous, which will complete the proof since $P(1/(2i-1))=\{\id\}$ for each $i\geq 1$. The continuity of $P$ on the interval $(0,1]$ follows from the continuity of $Q$, so it remains to show that $P(t)\rightarrow\{\id\}$ as $t\rightarrow 0$. This however follows from noting that for all $i\geq 1$ and $t<1/(2i-1)$, any $\varphi\in P(t)-\{\id\}$ is a hyperbolic isometry with translation length at least $\inf_{j\geq i}L_j$. Since $L_i\rightarrow\infty$ and therefore $\inf_{j\geq i}L_j\rightarrow\infty$ as $i\rightarrow\infty$, non-trivial isometries of $P(t)$ cannot accumulate anywhere in $\Isom^+(\BH^2)$, which completes the proof.
\end{proof}

\begin{example}[Non-example omitting condition (2) of the \hyperlink{thm:combo}{Combination Theorem}]
\label{ex:non}
    There exists a collection $\CP$ of paths $P:[0,1]\longrightarrow\CD_2$ such that each $P\in\CP$ has an assigned non-empty closed subset $A_P\subset\BH^2$ satisfying the following.
    \begin{enumerate}[a)]
        \item There exists a $\delta>0$ such that for all $P\neq P'$, $d(A_P,A_{P'})>\delta$ and the closures of $A_P$ and $A_{P'}$ in $\overline{\BH^2}$ are disjoint.
        \item There exists some $D>0$ and $P_*\in\CP$ such that $d(A_{P_*}, A_P)<D$ for infinitely many $P\in\CP$.
        \item For all $s\in [0,1]$, $P\in\CP$, and $\psi_s\in P(s)-\{\id\}$, \[\psi_s(\BH^2-A_P)\subset A_P.\]
        \item The map $\Pi:[0,1]\longrightarrow\CD_2$ defined by \[\Pi(s):=\left<\bigcup_{P\in\CP}P(s)\right>\] is not continuous.
    \end{enumerate}
\end{example}

\begin{proof}

    Let $\alpha\subset\BH^2$ be a geodesic, and let $A$ be one of the closed hyperbolic half-spaces that it bounds. Choose $B\subset A$ to be another closed hyperbolic half-space such that $B$ has positive distance from $\alpha$. Fix a hyperbolic isometry $\gamma$ with axis $\alpha$ such that $C:=\gamma^{-1}(B)$ and $B$ also have positive distance from one another; the choice of such a $\gamma$ is allowed by the condition that $B$ has positive distance from $\alpha$. Let $\psi\in\Isom^+(\BH^2)$ be an isometry that takes the exterior of $B$ to the interior of $C$. See Figure \ref{fig:Counterexample} for a depiction of these sets.\par

    \begin{figure}[h]
        \centering
        \includegraphics[width=0.35\linewidth]{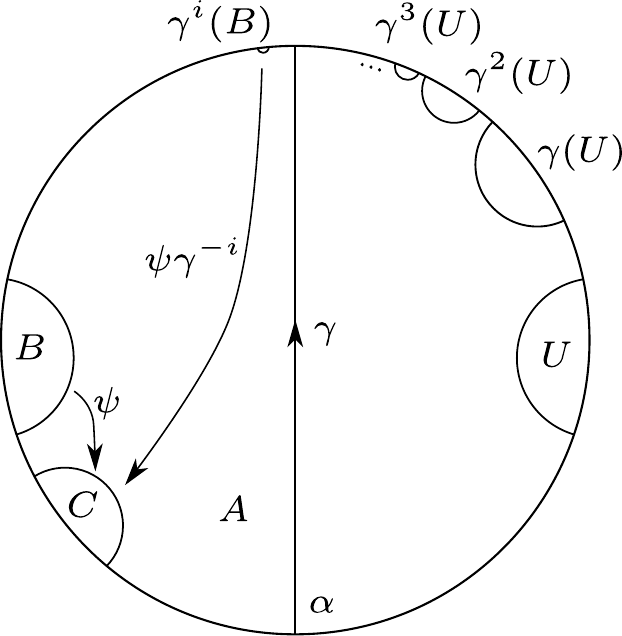}
        \caption{A depiction of the subsets and isometries appearing in Example \ref{ex:non}.}
        \label{fig:Counterexample}
    \end{figure}

    We now consider the family of isometries $\{\psi\gamma^{-i}\}_{i\geq 1}$. Note that for all $i\geq 1$, we have 
    \begin{equation}
    \label{eq:counterex ping pong}
        (\psi\gamma^{-i})(\BH^2-\gamma^i(B))=\mathrm{int}(C).
    \end{equation}
    Since $d(\gamma^i(B),C)\rightarrow\infty$ as $i\rightarrow\infty$, the isometries $\psi\gamma^{-i}$ are hyperbolic with translation length going to $\infty$, as $i\rightarrow\infty$. Thus, the construction in Example \ref{ex:trivial infinitely often} produces a path $P_*:[0,1]\longrightarrow\CD_2$ (there just called $P$) such that $P_*(1/2i)=\langle\psi\gamma^{-i}\rangle$ for all $i\geq 1$. Additionally, since $\gamma^i(B)\subset A$ for all $i$, it follows from Equation \ref{eq:counterex ping pong} and the construction of the path $P_*$ that 
    \begin{equation}
        \label{eq:PingPong for P}
        g(\BH^2-A)\subset A, \hspace{.2in} \forall\;t\in[0,1], g\in P_*(t)-\{\id\}.
    \end{equation}

    Informally, the failure of continuity as in (d) will be arranged by playing the path $P_*$ off of a set of constant paths constructed in $\BH^2-A$. More precisely, let $U\subset\BH^2-A$ be a closed hyperbolic half-space such that $U,\gamma(U),$ and $A$ all have positive distance from one another -- for example, $U$ may be the reflection of $B$ across the geodesic $\alpha=\partial A$. Let $\varphi\in\Isom^+(\BH^2)$ be an isometry such that \[\varphi^i(\BH^2-U)\subset U\] for all $i\neq 0$.\par

    For all $i\geq 0$, let $Q_i:[0,1]\longrightarrow\CD_2$ be the constant path at the subgroup $\langle \gamma^i\varphi\gamma^{-i}\rangle$. Note that for all $i\in\BZ$ and $j\neq 0$ we have
    \begin{equation}
        \label{eq:PingPong for Q}
        (\gamma^i\varphi\gamma^{-i})^j(\BH^2-\gamma^i(U))=(\gamma^i\varphi^j)(\BH^2-U)\subset\gamma^i(U).
    \end{equation}

    We now define $\CP$ to be the family of paths in $\CD_2$ consisting of $P_*$ and the paths $Q_i$ for $i\geq 0$. Defining $A_{P_*}:=A$ and $A_{Q_i}:=\gamma^i(U)$, the choice of $U$ implies that conditions (a) and (b) of the example statement are satisfied. In particular, note that \[d(A_{P_*},A_{Q_i})=d(A,U)\] for all $i$, so condition (b), the negation of condition (2) of the \hyperlink{thm:combo}{Combination Theorem}, is satisfied. Additionally, Equations \ref{eq:PingPong for P} and \ref{eq:PingPong for Q} imply that condition (c) of the example holds.\par

    It remains to check that the map $\Pi:[0,1]\longrightarrow\CD_2$ appearing in (d) is not continuous at $0$. First, note that $P_*(0)=\{\id\}$, and therefore $\Pi(0)=\langle\gamma^i\varphi\gamma^{-i}\;|\;i\geq 0\rangle$. Our goal is therefore to show that the groups $\Pi(t)$ accumulate on an isometry outside this group $\Pi(0)$. Indeed, note that for all $i\geq 1$, we have \[\psi\varphi\psi^{-1}=(\psi\gamma^{-i})(\gamma^i\varphi\gamma^{-i})(\psi\gamma^{-i})^{-1}.\] Thus, since $P_*(1/2i)=\langle\psi\gamma^{-i}\rangle$ and $\gamma^i\varphi\gamma^{-i}\in Q_i(t)$ for all $t$ (hence, in particular for $t=1/2i$), we see that $\psi\varphi\psi^{-1}$ is an element of $\Pi(1/2i)$ for all $i$, and is therefore accumulated on by the groups $\Pi(t)$ as $t\rightarrow 0$. Noting that $\psi\varphi\psi^{-1}\not\in\Pi(0)$, for example since $g(A)\cap A=\varnothing$ for all $g\in\Pi(0)-\{\id\}$, completes the proof.
\end{proof}

\section{The binary path}
\label{sec:binary path}

\subsection{Specifying homology}
\label{subsec:specifying homology}

For the construction of the Binary Path in Theorem \ref{thm:binary path}, we require a few topological lemmas that allow us to construct $3$-manifolds with specified homology torsion. We will use Dehn surgery to accomplish this.\par

Let $M$ be a compact orientable $3$-manifold containing a knot $K$ -- that is, $K$ is the image of an embedding $S^1\hookrightarrow M$ -- and let $\nu(K)\subset M$ be a tubular neighborhood of $K$. Assume that $K$ is homologically trivial in $M$: in this case, the torus $\partial\nu(K)$ has a canonical\footnote{In this setting, $\lambda$ may be called a ``surface longitude," as it is constructed by pushing the boundary of a Seifert surface into the knot exterior boundary (see e.g. \cite{Kegel}).} longitude $\lambda$ that is trivial in the first homology group $H_1(M-\mathrm{int}(\nu(K)),\BZ)$. As usual, $\partial\nu(K)$ also has a canonical meridian $\mu$ that is compressible in $\nu(K)$. Up to reversing the orientations of $\mu$ and $\lambda$, coprime integers $p,q$ determine the $3$-manifold $M^K_{p,q}$ obtained by $p/q$ \textit{Dehn surgery on $M$ along $K$}. Specifically, $M^K_{p,q}$ is constructed by gluing a solid torus to $\partial\nu(K)$ in $M-\mathrm{int}(\nu(K))$ such that the meridian of the solid torus maps to the curve $p\mu+q\lambda$ on $\partial \nu(K)$. See Rolfsen \cite[Ch. 9F]{Rolfsen} for more details.\par

A standard argument using the Mayer-Vietoris sequence gives the following lemma. See \cite[Lem. 3.4]{FPS-Surgery} for details.

\begin{lemma}
\label{lemma:prescribing homology}
    For a compact orientable $3$-manifold $M$ and a homologically trivial knot $K$, \[H_1(M^K_{p,q},\BZ)\cong H_1(M,\Z)\oplus\BZ/p\BZ\] for all coprime $p,q$.
\end{lemma}

The main tool that we need from this section is the following application of the previous lemma to hyperbolic 3-manifolds.

\begin{lemma}[Specifying homology torsion]
\label{lemma:realizing prescribed homology}
    For all primes $p$ sufficiently large, there exists an infinite volume convex cocompact hyperbolic $3$-manifold $M_p$ with first homology \[H_1(M_p,\BZ)\cong\BZ^2\oplus\BZ/p\BZ.\]
\end{lemma}

\begin{proof}

    Let $N$ denote an open genus-$2$ handlebody. Fix a knot $J$ in $N$ that is homologically trivial, but homotopically non-trivial. For example, $J$ could be in the homotopy class of the loop $a^{-1}b^{-1}ab\in\pi_1(N)=\langle a,b\rangle$. The main theorem in \cite{ExcellentMyers} combined with Thurston's hyperbolization theorem \cite{THUI} implies that $J$ is homotopic to a knot $K$ such that the knot complement $N-K$ is hyperbolizable. For each prime $p$, Lemma \ref{lemma:prescribing homology} implies that \[H_1(N_{p,1}^K,\BZ)\cong\BZ^2\oplus \BZ/p\BZ,\] and work of Comar \cite{ComarThesis} implies that $N_{p,1}^K$ admits an infinite volume convex cocompact hyperbolic structure for $p$ sufficiently large.
\end{proof}

\subsection{Construction of the Binary Path}

\label{subsec:binary construction}

Our first example here will be a ``building block" for the \hyperlink{thm:binary}{Binary Path}. We will describe a general method for producing isomorphism bump paths as defined in Definition \ref{def:Isom bump path}, and the \hyperlink{thm:binary}{Binary Path} will be constructed as a ``combination" of these isomorphism bump paths, using the \hyperlink{thm:combo}{Combination Theorem}.

\begin{example}[Isomorphism bump path]
\label{ex:isom bump}
    Suppose $\Gamma\leq\Isom^+(\BH^3)$ is discrete and torsion free, and that there exists a hyperbolic half-space $A\subset\BH^3$ such that \[\psi(\BH^3-A)\subset A\hspace{.4in}\text{for all }\psi\in\Gamma-\{\id\}.\] Then, for any open interval $I\subset\BR$, there exists a path $B:\BR\longrightarrow\CD_3$ such that \[B(s)\cong
    \begin{cases}
        \Gamma, &s\in I\\
        \{\id\}, &\text{else}.
    \end{cases}\]
    Furthermore, $\psi_s(\BH^3-A)\subset A$ for all $s\in \BR$ and $\psi_s\in B(s)-\{\id\}$.
\end{example}

\begin{proof}
    It suffices to assume $I=(0,1)$, as the general case then follows by a reparameterization. Working in the upper half-space model for $\BH^3$, replacing $\Gamma$ with an $\Isom^+(\BH^3)=\PSL_2\C$ conjugate allows us to assume that the hyperbolic half-space $A$ is the convex hull of the disk \[D:=\{z\in\C\;|\;|z|\leq 1\}\subset\partial\overline{\BH^3},\] where $\partial\overline{\BH^3}=\BC\cup\{\infty\}$ is the boundary of $\BH^3$ ``at infinity". For $t>0$, consider the isometry 
    \[\varphi_t:=
    \begin{bmatrix}
        t & 0\\
        0 & 1/t
    \end{bmatrix}\in\PSL_2\C,\]
    which acts as a M\"obius transformation on $\BC$ by the map $z\mapsto t^2z$. In particular, we have $\varphi_t(D)=t^2D\subset\BC$, so $\varphi_t(A)=t^2 A=:\hull(t^2D)$.\par

    We now define a map $B:\BR\longrightarrow\CD_3$ by 
    \[B(t)=
    \begin{cases}
    \{\id\}, &t=0,1\\
    \varphi_{2t}\Gamma\varphi_{2t}^{-1}, &0<t\leq 1/2,\\
    \varphi_{2-2t}\Gamma\varphi_{2-2t}^{-1}, &1/2<t<1.
    \end{cases}\] Informally, one should think of $B$ as bringing in the group $\Gamma$ from infinity at $t=0$, and sending it back out to infinity at $t=1$. \par
    
    We claim now that $B$ is a (continuous) path in $\CD_3$. Since $\varphi_t$ varies continuously for $t>0$, it suffices to show that $B(t)\rightarrow\{\id\}$ in $\CD_3$ as $t\rightarrow 0$ and as $t\rightarrow 1$. Note that for any $t>0$, \begin{equation}
    \label{eqn:small ball}
        (\varphi_t\psi\varphi_t^{-1})(\BH^3-t^{2}A)=\varphi_t\psi(\BH^3-A)\subset\varphi_t(A)=t^2A,\hspace{.2in}\forall\psi\in\Gamma-\{\id\}.
    \end{equation}
    For any compact set $K\subset\BH^3$, we have $t^2 A\cap K=\varnothing$ for $t$ sufficiently close to $0$. It follows that the non-trivial elements of $B(t)$ cannot accumulate anywhere in $\PSL_2\C$, as $t\rightarrow 0$ or $t\rightarrow 1$, which proves the claim. 
\end{proof}

Similar examples to the path $B$ were constructed by the author in Examples C.1 and C.2 in \cite{ZevenbergenPaths1}. Additionally, note that $B$ is an isomorphism bump path as defined in Definition \ref{def:Isom bump path}, as the requisite continuous map $(0,1)\longrightarrow\Hom(\Gamma,\Isom^+(\BH^3))$ can be produced by conjugating the identity homomorphism.\par

We now define a collection of hyperbolic half-spaces, which will become the collection of disjoint closed subsets used when citing the \hyperlink{thm:combo}{Combination Theorem} to justify the continuity of the \hyperlink{thm:binary}{Binary Path}. For each $n,m\in\BZ$, define \[D_{n,m}:=\{z\in\BC\;|\;|z-(n+im)|\leq 1/3\},\hspace{.2in}\text{and}\hspace{.2in}A_{n,m}:=\hull(D_{n,m})\subset\BH^3,\] so that $A_{n,m}$ is the closed hyperbolic half-space accumulating on the radius $1/3$ disk $D_{n,m}\subset\BC\subset\partial\overline{\BH^3}$ centered at $n+im$.\par

\begin{lemma}
\label{lemma:conditions 1 and 2}
    For any $D>0$, if $\Delta:=|(n,m)-(n',m')|$ is sufficiently large, then \[d(A_{n,m},A_{n',m'})>D.\] Further, there exists $\delta>0$ such that $d(A_{n,m},A_{n',m'})>\delta$ if $(n,m)\neq(n',m')$.
\end{lemma}

\begin{proof}
    It follows, for example from Farb's analysis of geodesics in ``electrified" hyperbolic spaces \cite{FarbRelativelyHyperbolic}, that if $\Delta$ is sufficiently large, the shortest path $\gamma$ between $A_{n,m}$ and $A_{n',m'}$ must intersect the height $1$ horosphere $T$ centered at infinity in a pair of points $x,y$. Further, there is a universal constant $\varepsilon$ such that the intersection points $x$ and $y$ must have distance at least $\Delta-\varepsilon$ from one another in the intrinsic Euclidean metric on $T$. We then have \[d(A_{n,m}, A_{n',m'})\geq d(x,y)\geq 2\tanh^{-1}\left(\frac{(\Delta-\varepsilon)^2}{(\Delta-\varepsilon)^2+4}\right)^{1/2},\] by \cite[Cor. A.5.8]{BEN}. The value on the right-hand side of the inequality goes to $\infty$ as $\Delta\rightarrow\infty$. This establishes the first statement of the lemma.\par

    It follows from the first statement that there exists a $\delta>0$ such that $d(A_{0,0}, A_{n,m})>\delta$ for all $(n,m)\neq (0,0)$. The disjoint union $\bigcup_{n,m\in\BZ}A_{n,m}$ is invariant under a group of isometries that acts transitively on the components of this union, allowing us to conclude the second statement of the lemma.
\end{proof}

This section closes with the construction of the Binary Path.

\vspace{.1in}

\noindent\textbf{{Theorem~\ref*{thm:binary path}}} (Binary Path)
\textit{There exists a path $\Phi:[0,1]\longrightarrow\CD_3$ such that $\Phi(t)$ is isomorphic to $\Phi(s)$ if and only if $t=s$.}

\begin{proof}
By Lemma \ref{lemma:realizing prescribed homology}, we can let $\{p_n\}_{n\geq 0}$ be a set of distinct primes for which there exists an infinite volume convex cocompact hyperbolic $3$-manifold $M_{p_n}=\Gamma_{p_n}\backslash\BH^3$ with 
\begin{equation}
\label{eqn: homology}
    H_1(M_{p_n},\BZ)\cong\BZ^2\oplus\BZ/p_n\BZ.    
\end{equation} Since each $\Gamma_{p_n}$ has non-empty domain of discontinuity, for any $m\in\BZ$ there is a $\PSL_2\C$ conjugate $\Gamma_{p_n}^{n,m}$ of $\Gamma_{p_n}$ such that \[\psi(\BH^3-A_{n,m})\subset A_{n,m},\hspace{.2in} \forall\psi\in\Gamma_{p_n}^{n,m}-\{\id\}.\]\par

For $n\geq 0$ and $-(2^n-1)\leq m\leq 2^n-1$, Example \ref{ex:isom bump} therefore allows us to define paths $B_{n,m}:\BR\longrightarrow\CD_3$ so that 

\begin{equation}
\label{eqn: B isom}
    B_{n,m}(t)\cong
    \begin{cases}
        \Gamma_{p_n}, &0\leq m\;\text{ and }\;{-1}/2< 2^{n+1}t-2m<1\\
        \Gamma_{p_n}, &m<0\;\text{ and }\;{-1}/2< 2^{n+1}t+2m<0\\
        \{\id\}, &\text{else}
    \end{cases}
\end{equation}

and for all $s\in\BR$ and $n,m$, we have
\begin{equation}
    \label{eqn:B ping pong}
    \psi_s(\BH^3-A_{n,m})\subset A_{n,m},\hspace{.2in}\forall\psi\in B_{n,m}(s)-\{\id\}.
\end{equation}\par

We now claim that the collection $\CP$ consisting of the paths $B_{n,m}$ for $n\geq 0$ and $-2^n+1\leq m\leq 2^n-1$ satisfy the hypotheses of the \hyperlink{thm:combo}{Combination Theorem}. Here, we associate to each path $B_{n,m}$ the closed subset $A_{n,m}\subset\BH^3$. It follows from Lemma \ref{lemma:conditions 1 and 2} that this collection $\CP$ satisfies conditions (1) and (2) of the \hyperlink{thm:combo}{Combination Theorem}. That condition (3) is satisfied follows immediately from Equation \ref{eqn:B ping pong}. Thus, it follows from the \hyperlink{thm:combo}{Combination Theorem} that the map $\Phi':\BR\rightarrow\CD_3$ defined by \[\Phi'(t):=\left\langle\bigcup_{\substack{n\geq 0\\-2^n+1\leq m\leq 2^n-1}}B_{n,m}(t)\right\rangle\] is a path. We define $\Phi:[0,1]\longrightarrow\CD_3$ to be the restriction of $\Phi'$ to the domain $[0,1]$. \par

It remains to check that if $\Phi(t)\cong\Phi(s)$, then $t=s$. For a Kleinian group $\Gamma$, let  $\abtor(\Gamma)$ denote the subgroup of finite order elements of the abelianization of $\Gamma$. What we will show is that if $\abtor(\Phi(t))\cong\abtor(\Phi(s))$, then $t=s$. \par

Note that since each $\Phi(t)$ splits as a free product of the groups $B_{n,m}(t)$, we have \[\abtor(\Phi(t))\cong \prod_{\substack{n\geq 0\\-2^n+1\leq m\leq 2^n-1}}\abtor(B_{n,m}(t))\] for all $t$. We see from Equations \ref{eqn: homology} and \ref{eqn: B isom} that for any $n,m$
\[\abtor(B_{n,m}(t))\cong
\begin{cases}
    \BZ/p_n\BZ, &m\geq 0\;\text{ and }\;{-1}/2<2^{n+1}t-2m<1\\
    \BZ/p_n\BZ, &m< 0\;\text{ and }\;{-1}/2<2^{n+1}t+2m<0\\
    \{\id\}, & \text{else}.
\end{cases}\]
Thus, $\abtor(\Phi(t))$ has \textit{exactly one} subgroup of order $p_n$ if and only if $0\leq 2^{n+1}t-2m<1$ for some $0\leq m\leq 2^n-1$. Indeed, $\abtor(\Phi(t))$ has \textit{at least one} subgroup of order $p_n$ if and only if ${-1}/2< 2^{n+1}t-2m<1$ for some $0\leq m\leq 2^n-1$, but if ${-1}/2< 2^{n+1}t-2m<0$, then $\abtor(\Phi(t))$ has a $(\BZ/p_n\BZ)^2$ subgroup.\par

For $t\in[0,1)$ and $n\geq 0$, set $a_n^t=0$ if $\abtor(\Phi(t))$ has exactly one subgroup of order $p_n$, and otherwise set $a_n^t=1$. We therefore recover a ``binary representation" for $t$, as \[t=\sum_{n\geq 0}\frac{a_n^t}{2^{n+1}}.\] This confirms that $t$ is uniquely determined by $\abtor(\Phi(t))$, since we also have $\abtor(\Phi(t))\cong\{\id\}$ if and only if $t=1$.
\end{proof}

\section{Decompositions of paths of convex cocompact groups}

\label{sec:decomp}

\subsection{Chabauty topology path machinery}
\label{subsec:path machinery}

We now summarize and expand upon the machinery developed by the author in \cite{ZevenbergenPaths1} to analyze paths in $\CD_n$, beginning with the following lemma.

\begin{lemma}[{\cite[Lem. 4.1]{ZevenbergenPaths1}}]
    \label{Unbhd}
    For a fixed $\Gamma\in\CD_n$ and $\psi\in\Gamma$, there exists a neighborhood $U_\psi\subset\Isom^+(\BH^n)$ such that if  $\Gamma_i\rightarrow\Gamma$ is a convergent sequence in $\CD_n$, $\Gamma_i\cap U_\psi$ contains exactly one element for $i$ sufficiently large.
\end{lemma}

For the remainder of the section, fix a path $P:I\longrightarrow\CD_n$, where $I\subset\BR$ is an interval, and let $\Gamma_t$ denote $P(t)$. Our next goal will be to describe a way to ``track" an element $\psi\in\Gamma_t$ through the path of subgroups. As we track elements, some may ``diverge to infinity." To make this precise, we let 
\[\overline{\Isom^+(\BH^n)}=\Isom^+(\BH^n)\cup\{\infty\}\] denote the one-point compactification of $\Isom^+(\BH^n)$. We say that any map $J:X\longrightarrow\overline{\Isom^+(\BH^n)}$ is \textit{finite valued} if $J(X)\subset\Isom^+(\BH^n)$.\par

The following proposition summarizes the results of Section 4.2 of \cite{ZevenbergenPaths1}, by the author.

\begin{prop}[{\cite[Sec. 4.2]{ZevenbergenPaths1}}]
\label{prop:J Facts}
    For all $s\in I$, there exists a unique family of maps \[J_{s,t}:\Gamma_s\longrightarrow\overline{\Isom^+(\BH^n)}\] such that for all  $\psi\in\Gamma_s$,
    \begin{enumerate}
        \item $J_{s,t}(\psi)\in\Gamma_t\cup\{\infty\}$ for all $t\in I$,
        \item $J_{s,s}(\psi)=\psi$,
        \item the map $t\mapsto J_{s,t}(\psi)\in\overline{\Isom^+(\BH^n)}$ is a continuous function of $t\in I$, 
        \item there exists an interval $I_\psi^s$ open in $I$ such that $J_{s,t}(\psi)\neq\infty$ if and only if $t\in I_\psi^s$, and
        \item if $t,t'\in I_\psi^s$, then $t'\in I_{J_{s,t}(\psi)}^t$ and $(J_{t,t'}\circ J_{s,t})(\psi)=J_{s,t'}(\psi)$.
    \end{enumerate}
\end{prop}

Before continuing, we will make a few remarks about this proposition. A point that will be important to emphasize for the present discussion is the continuity of the path $t\mapsto J_{s,t}(\psi)$ given by Proposition \ref{prop:J Facts}(3). Informally, this tells us that we can use the path $J_{s,t}(\psi)$ to track an element $\psi\in\Gamma_s$ through the groups $\Gamma_t$ for $t$ close to $s$; indeed, for $t$ close to $s$, $J_{s,t}(\psi)$ is constructed to be the unique element of $U_\psi\cap\Gamma_t$, where $U_\psi\subset\Isom^+(\BH^n)$ is the subset given by Lemma \ref{Unbhd}. Additionally, this continuity implies that the only way that we can ``lose track" of an element of $\Gamma_s$ (i.e., fail to match it to an element of $\Gamma_t$) is if the path $t\mapsto J_{s,t}(\psi)$ diverges to $\infty$ in $\overline{\Isom^+(\BH^n)}$ as $t$ approaches the frontier of the interval $I_\psi^s$. \par

We also record the following lemma.

\begin{lemma}[{\cite[Prop. 4.6, Cor. 4.7]{ZevenbergenPaths1}}]
\label{lem:injhomo}
    For any subgroup $H\leq\Gamma_s$ and $s,t\in I$, $J_{s,t}|_H$ is an injective homomorphism into $\Gamma_t$ if and only if $J_{s,t}|_H$ is finite valued. If $H$ is finitely generated, then $J_{s,t}$ is finite valued for $t$ sufficiently close to $s$.
\end{lemma}

Observe that the condition that $J_{s,t}$ is finite valued is also equivalent to having $t\in I_\psi^s$ for all $\psi\in\Gamma_s$, where the intervals $I_\psi^s$ are those appearing in Proposition \ref{prop:J Facts}. Additionally, if $s\in I$ and $I'\subset I$ is an interval containing $s$ such that $J_{s,t}$ is finite valued for all $t\in I'$, then it follows from Lemma \ref{lem:injhomo} that we can define a map $I'\longrightarrow\Hom(\Gamma_s,\Isom^+(\BH^n))$ by $t\mapsto J_{s,t}$. Furthermore, it follows from the continuity in Proposition \ref{prop:J Facts}(3) that this map is continuous, where $\Hom(\Gamma_s,\Isom^+(\BH^n))$ is equipped with the compact-open topology.\par

Though we will in general lose the continuity given by Proposition \ref{prop:J Facts}(3), it will be useful to define a variant of each map $J_{s,t}$ whose image is contained in $\Isom^+(\BH^n)$. Accordingly, for each $s,t\in I$, we define the map $J_{s,t}^*:\Gamma_s\longrightarrow\Isom^+(\BH^n)$ by \[J_{s,t}^*(\psi)=
\begin{cases}
    J_{s,t}(\psi), &t\in I_\psi^s\\
    \id,& \text{else}.
\end{cases}\]
In particular, we see that $J_{s,t}^*(\Gamma_s)=J_{s,t}(\Gamma_s)\cap\Isom^+(\BH^n)$.

\begin{lemma}
\label{Lemma-SubgroupConvg}
    For any $s,t\in I$ and subgroup $H\leq\Gamma_s$, $J_{s,t}^*(H)$ is a subgroup of $\Gamma_t$. Further, the map $P':I\longrightarrow\CD_n$ defined by $P'(t)=J_{s,t}^*(H)$ is continuous.
\end{lemma}

\begin{proof}

    Fix $s,t\in I$. Define \[H'=\{\psi\in H\;|\;t\in I_\psi^s\},\] the set of elements $\psi\in H$ for which $J_{s,t}(\psi)\neq\infty$.  For all $\psi,\varphi\in H$, we have $I_\psi^s=I_{\psi^{-1}}^s$ and $I_\psi^s\cap I_\varphi^s\subset I_{\psi\varphi}^s$, so $H'$ is a subgroup of $\Gamma_s$. Observe that $J_{s,t}^*(H)=J_{s,t}(H')$ and Lemma \ref{lem:injhomo} implies that $J_{s,t}|_{H'}$ is an injective homomorphism. The image of $J_{s,t}|_{H'}$ lies in $\Gamma_t$, so we confirm that $J_{s,t}^*(H)\leq\Gamma_t$.\par

    To address continuity in the Chabauty topology, we fix $s,t\in I$ and will prove that \[J_{s,t'}^*(H)\rightarrow J_{s,t}^*(H)\] as $t'\rightarrow t$. To check the first condition for convergence in the Chabauty topology as stated in Definition \ref{def:Chabauty}, suppose there exist sequences $\{t_k\}_{k\geq 1}\subset I$ and $\{\psi_k\}_{k\geq 1}\subset H$ such that $t_k\rightarrow t$ and $J_{s,t_k}^*(\psi_k)\rightarrow \varphi\in\Isom^+(\BH^n)$ as $k\rightarrow\infty$. If $\varphi=\id$, then indeed $\varphi\in J_{s,t}^*(H)$. Suppose, then, that $\varphi\neq\id$, in which case \[J_{s,t_k}^*(\psi_k)=J_{s,t_k}(\psi_k)\neq\id\] for $k$ sufficiently large. Since $\Gamma_{t_k}\rightarrow \Gamma_t$ as $k\rightarrow\infty$, we must have $\varphi\in\Gamma_t$. Letting $U_\varphi\subset\Isom^+(\BH^n)$ be the neighborhood of $\varphi$ given by Lemma \ref{Unbhd} with respect to $\Gamma_t$, we must have $J_{s,t_k}(\psi_k)\in U_\varphi$ for $k$ sufficiently large. Additionally, for large $k$, $J_{t,t_k}(\varphi)$ is constructed to be the unique element of $\Gamma_{t_k}$ contained in $U_\varphi$. Thus, after potentially removing finitely many terms from the sequences and reindexing, we must have \[J_{s,t_k}(\psi_k)=J_{t,t_k}(\varphi)\] for all $k$. Note that Proposition \ref{prop:J Facts}(5\&2) says that $(J_{t_k,t}\circ J_{t,t_k})(\varphi)=\varphi$. Thus, several more applications of Proposition \ref{prop:J Facts}(5\&2) imply that $t\in I_{\psi_k}^s$ and \[J_{s,t}(\psi_k)=(J_{t_k,t}\circ J_{s,t_k})(\psi_k)=J_{t_k,t}(J_{t,t_k}(\varphi))=\varphi\] for all $k$. This confirms that $\varphi\in J_{s,t}^*(H)$, as was to be shown. For clarity, we note that it now follows from Lemma \ref{lem:injhomo} that $\psi_k=\psi_{k'}$ for all $k,k'$, though we do not need this fact.\par

    To check the second condition for Chabauty topology convergence, fix an arbitrary element $J_{s,t}^*(\psi)$, where $\psi\in H$. Since $\id\in J_{s,t'}^*(H)$ for every $t'$, we may assume that $J_{s,t}^*(\psi)\neq\id$, in which case $t\in I_{\psi}^s$. The proof is then finished by noting that $I_{\psi}^s$ is open by Proposition \ref{prop:J Facts}(4), and applying the continuity given by Proposition \ref{prop:J Facts}(3). 
\end{proof}

The next lemma roughly says that these maps $J_{s,t}^*$ fully describe an isomorphism bump path.

\begin{lemma}
    \label{lem:bump vs J}
    If $B:I\longrightarrow\CD_n$ is an isomorphism bump path and $s\in I$ is such that $B(s)\neq\{\id\}$, then \[B(t)=J_{s,t}^*(B(s))\] for all $t\in I$.
\end{lemma}

\begin{proof}
    Since $B(s)\neq\{\id\}$, Definition \ref{def:Isom bump path} of an isomorphism bump path says that there exists a subinterval $I'\subset I$ and a continuous map $R:I'\longrightarrow\Hom(B(s),\Isom^+(\BH^n))$ such that each $R(t)$ is an injective homomorphism and 
    \[B(t):=
    \begin{cases}
        (R(t))(B(s)), &t\in I' \\
        \{\id\}, &t\in I-I',
    \end{cases}\]
    with $R(s)$ the identity homomorphism. For each $\psi\in B(s)$, the map $t\mapsto (R(t))(\psi)$ is a continuous path in $\Isom^+(\BH^n)$. Since $(R(s))(\psi)=\psi$ and $(R(t))(\psi)\in B(t)$ for all $t\in I'$, Proposition 4.3 in \cite{ZevenbergenPaths1}, a strong version of the uniqueness in Proposition \ref{prop:J Facts}, implies $J_{s,t}^*(\psi)=(R(t))(\psi)$ for all such $t$. This confirms that $J_{s,t}^*(B(s))=B(t)$ for all $t\in I'$.\par

    For $t\not\in I'$, we have \[J_{s,t}^*(B(s))\leq B(t)=\{\id\}\] by Lemma \ref{Lemma-SubgroupConvg}, and hence again $J_{s,t}^*(B(s))=\{\id\}$.
\end{proof}

\subsection{Proof of decomposition theorem}
\label{subsec:decomp thm}

We start by establishing some facts about loxodromic isometries of $\BH^3$. For a loxodromic $\psi\in\Isom^+(\BH^3)$, let $A_\psi$ denote its unique invariant axis on which $\psi$ translates some length $\ell(\psi)$. Let \[T_\eta(\psi):=\{x\in\mathbb{H}^3\;|\; d(x,\psi\cdot x)\leq \eta\}\] be the set of points moved by $\psi$ a distance at most $\eta>0$. For $M$ a hyperbolic $3$-manifold, a subset $S\subset M$, and $D>0$, define \[\mathcal{N}_D(S):=\{x\in M:d(x,S)\leq D\},\] the closed radius $D$ neighborhood of $S$.

\begin{lemma}
    \label{lemma-LoxodromicFacts}
    Let $\psi\in\Isom^+(\BH^3)$ be a loxodromic isometry. 
    \begin{enumerate}
        \item There exists $R=R(\ell(\psi))$ such that $T_1(\psi)\subset\mathcal{N}_R(A_\psi)$, and $R(\ell(\psi))$ may be taken to be a decreasing function of $\ell(\psi)$.
        \item For all $\eta>0$, there exists $D=D(\eta)$ such that $T_\eta(\psi)\subset\CN_D(T_1(\psi)\cup A_\psi)$.
    \end{enumerate}
\end{lemma}

We emphasize that the constant $D$ does not depend on $\ell(\psi)$, only on $\eta$. Additionally, in Lemma \ref{lemma-LoxodromicFacts}(2) we consider the distance from $x$ to the set $T_1(\psi)\cup A_\psi$, since $T_1(\psi)=\varnothing$ if $\ell(\psi)>1$. Using $T_1(\psi)$ in the lemma statement rather than some other $T_\varepsilon(\psi)$ is not significant, and replacing $1$ with a different constant just leads to different values for $R$ and $D$. 

\begin{proof}[Proof of Lemma \ref{lemma-LoxodromicFacts}:]

    Recall that for any $\eta>0$ and loxodromic isometry $\varphi$ of $\BH^3$ with $\ell(\varphi)\leq\eta$, $T_\eta(\varphi)$ is a metric neighborhood of $A_\varphi$. We can then define $R_\varphi^\eta\geq 0$ to be the value such that \[T_\eta(\varphi)=\CN_{R_\varphi^\eta}(A_\varphi).\] If $\varphi$ has no rotation part (that is, $\varphi$ is conjugate into $\Isom^+(\BH^2)\leq\Isom^+(\BH^3)$), then a standard computation with a Saccheri quadrilateral gives that
    \begin{equation}
    \label{eqn-RVarphi}
        R_{\varphi}^\eta=\arcosh\left(\frac{\sinh(\eta/2)}{\sinh(\ell(\varphi)/2)}\right).
    \end{equation}

    (1): Let $\psi$ be the given loxodromic isometry. If $\ell(\psi)>1$, we may take $R=0$, so it suffices to assume that $\ell(\psi)\leq 1$. Fix a loxodromic $\varphi\in\Isom^+(\BH^3)$ so that $A_\varphi=A_\psi$, $\ell(\varphi)=\ell(\psi)$, and $\varphi$ has no rotation part. One can check that $T_1(\psi)\subset T_1(\varphi)$, so Equation \ref{eqn-RVarphi} implies that it suffices to define \[R=R(\ell(\psi)):=\arcosh\left(\frac{\sinh(1/2)}{\sinh(\ell(\psi)/2)}\right).\]
    
    (2): For an arbitrary loxodromic $\varphi\in\Isom^+(\BH^3)$ such that $\ell(\varphi)\leq\eta$, we define $D_\varphi^\eta$ to be the infimal value $D_\varphi^\eta$ such that $T_\eta(\varphi)\subset\CN_{D_\varphi^\eta}(T_1(\varphi)\cup A_\varphi)$. If, additionally, $1\leq\eta$, then $D_\varphi^\eta$ is the distance between $T_1(\varphi)\cup A_\varphi$ and $\partial T_\eta(\varphi)$. \par
    
    Let $\psi$ be the given loxodromic isometry and fix $\eta>0$. Similar to above, it suffices to assume $\eta\geq \max\{1,\ell(\psi)\}$, else we can set $D=0$. As in the proof of the first statement, choose a loxodromic $\varphi\in\Isom^+(\BH^3)$ with $A_\varphi=A_\psi$ and $\ell(\varphi)=\ell(\psi)$ so that $\varphi$ has no rotation part. One can confirm that $D_{\psi}^\eta\leq D_\varphi^\eta$, so our goal will be to find an upper bound for $D_\varphi^\eta$ that depends only on $\eta$.\par

    Since $D_{\varphi}^\eta$ is a continuous function of $\ell(\varphi)$, it suffices to show that $\lim_{\ell(\varphi)\rightarrow 0}D_\varphi^\eta$ exists. Note that for $0<\ell(\varphi)\leq 1$, we have \[D_\varphi^\eta=R_\varphi^\eta-R_\varphi^1=\arcosh\left(\frac{\sinh(\eta/2)}{\sinh(\ell(\varphi)/2)}\right)-\arcosh\left(\frac{\sinh(1/2)}{\sinh(\ell(\varphi)/2)}\right).\] For any functions $f,g$ such that $\lim_{x\rightarrow 0}f(x)=\lim_{x\rightarrow 0}g(x)=\infty$ and $\lim_{x\rightarrow 0}(\ln f(x)-\ln g(x))$ exists, one has \[\lim_{x\rightarrow 0}(\arcosh f(x)-\arcosh g(x))=\lim_{x\rightarrow 0}(\ln f(x)-\ln g(x)).\] Thus, since \[\lim_{\ell(\varphi)\rightarrow 0}\ln\left(\frac{\sinh(\eta/2)}{\sinh{(\ell(\varphi)/2)}}\right)-\ln\left(\frac{\sinh(1/2)}{\sinh(\ell(\varphi)/2)}\right) = \ln\left(\frac{\sinh(\eta/2)}{\sinh(1/2)}\right)\] we confirm that $D_\varphi^\eta$ is uniformly bounded in terms of $\eta$, and the claim follows.
\end{proof}

For a Kleinian group $\Gamma$ and $\eta>0$, we define \[\BH^3_{\leq\eta}(\Gamma):=\{x\in\mathbb{H}^3\;|\;\exists \psi\in\Gamma-\{\id\}\text{ such that } d(x,\psi\cdot x)\leq\eta \}=\bigcup_{\psi\in \Gamma-\{\id\}}T_\eta(\psi).\] For the hyperbolic $3$-manifold $M=\Gamma\backslash\BH^3$, we define \[M_{\leq\eta}:=\Gamma\backslash(\BH^3_{\leq\eta}(\Gamma))\hspace{.2in}\text{ and }\hspace{.2in}M_{>\eta}:=M-M_{\leq\eta}.\] Equivalently, $M_{>\eta}$ is the set of points $p\in M$ such that the open radius $\eta/2$-ball $B_M(p,\eta/2)\subset M$ is isometric to such a ball in $\BH^3$.\par

In what follows, $\Lambda(\Gamma)$ will denote the \textit{limit set} of the Kleinian group $\Gamma$ (see \cite{MAT}). For a subset $X\subset \overline{\BH^3}$, $\hull(X)$ will denote the \textit{convex hull} of $X$ in $\BH^3$: that is, $\hull(X)$ is the smallest convex subset of $\BH^3$ such that $X\subset\overline{\hull(X)}$. To simplify notation, we denote \[\hull(\Gamma):=\hull(\Lambda(\Gamma)).\]

\begin{defn}[$\TCH$ and $\TCC$]
    For a Kleinian group $\Gamma$, the \textbf{thick convex hull} of $\Gamma$ is the set \[\TCH(\Gamma)=\mathrm{CH}(\Gamma)\cup \BH^3_{\leq1}(\Gamma).\] The \textbf{thick convex core} of $M=\Gamma\backslash\mathbb{H}^3$, denoted $\TCC(M)$, is the projection of $\TCH(\Gamma)$ to $M$.
\end{defn}

This terminology is used in \cite{IANRANK} to refer to what would be called $\CN_1(\mathrm{CH}(\TCH(\Gamma)))$ here. Note that $\TCC(M)=\CC(M)\cup M_{\leq1}$. The next lemma gives the key property that we will need of the thickened convex core. Roughly, since we have included all sufficiently thin parts in $\TCC(M)$, all points in $M$ of bounded injectivity radius are within a bounded distance of $\TCC(M)$.

\begin{lemma}
\label{Lemma-TCCNeighborhood}
    For all $\eta>0$, there exists $D=D(\eta)>0$ such that if $M$ is a hyperbolic $3$-manifold, $M_{\leq\eta}\subset\mathcal{N}_D(\TCC(M))$. 
\end{lemma}

\begin{proof}
    Letting $D=D(\eta)$ be the constant given by Lemma \ref{lemma-LoxodromicFacts}(2), it follows at once that  \[\BH^3_{\leq\eta}(\Gamma)\subset\CN_D(\TCH(\Gamma)),\] since $T_1(\psi)\cup A_\psi\subset \TCH(\Gamma)$.
\end{proof}

Fix $M=\Gamma\backslash\BH^3$ a convex cocompact hyperbolic $3$-manifold such that $\Gamma$ is not cyclic. Additionally, assume that $\Gamma$ is not conjugate to a subgroup of $\Isom^+(\BH^2)\leq\Isom^+(\BH^3)$, so that $\CC(M)$ is a $3$-manifold with boundary. Here, $\CC(M)$ is the \textit{convex core} of $M$, which is the image in $M$ of the smallest non-empty convex $\Gamma$-invariant subset of $\BH^3$.\par

Fix $\xi>0$ less than the $2$ and $3$-dimensional Margulis constants, and assume that $A$ is a component of $(\partial \CC(M))_{\leq\xi}$, the $\xi$-thin part of the surface $\partial \CC(M)$. $A$ is an annulus with boundary curves $\alpha_1$ and $\alpha_2$ of length between $\xi$ and $2\xi$. Suppose that $\alpha_1$ (and hence $\alpha_2$) is compressible in $\CC(M)$. Then, each $\alpha_i$ bounds a disk $D_i\subset \CC(M)$ of diameter at most $\xi$ with $D_1\cap D_2=\varnothing$. The $2$-sphere $D_1\cup A\cup D_2$ in $M$ must bound a ball $h\subset \CC(M)$. After Bowditch in \cite{BOW}, we will call $h$ a \textit{$\xi$-handle}. Note that we can assume that every point in $h$ is contained in a compressing disk $(D,\partial D)\subset (h, A)$ with diameter at most $\xi$. By the choice of $\xi$, all of the $\xi$-handles constructed as above can be made disjoint. We let $W=W(\xi)\subset \CC(M)$ denote the union of all $\xi$-handles, and define $C=C(\xi)\subset \partial W$ by \[C(\xi):=\bigcup_{h\text{ a }\xi\text{-handle}} D_1^h\cup D_2^h,\] where $D_1^h,D_2^h$ are the bounded diameter compressing disks appearing in the construction of $h$.\par

The proof of the main result\footnote{In \cite[Thm. 1.1]{BOW}, Bowditch claims that in fact $\CC(M)\subset M_{\leq\eta}\cup W$. For any $\eta>0$, there are, however, hyperbolic genus-$2$ handlebodies $M$ for which $M_{\leq\eta}=\varnothing$, yet $\CC(M)$ cannot be a disjoint union of $1$-handles. The modified conclusion (which suffices to prove the main result of \cite{BOW}) follows from filling $2$-spheres such as those in $\CC(M)- W$ with bounded diameter $3$-balls, as Bowditch does via \cite[Lem. 4.1]{BOW}.} of Bowditch in \cite{BOW} implies that there exists $\eta>0$ depending only on $\xi$ and the topology of $M$ such that \[\CC(M)\subset M_{\leq\eta}\cup W\cup\CN_\eta(C).\] In this statement, one should think of $\eta$ as being quite large. We also note that the number of $\xi$-handles that $M$ contains is at most $\rank(M)$, the size of a minimal generating set for $\pi_1(M)$. It follows that $C$ contains at most $2\rank(M)$ compressing disks. \par

The following proposition provides a way of recognizing a free factor of the fundamental group of a hyperbolic $3$-manifold.

\begin{prop}
\label{prop-CCEmbedFactor}
    For any convex cocompact hyperbolic $3$-manifold $M=\Gamma\backslash \BH^3$, there exists $L>0$ depending only on the topology of $M$ as follows. Fix a subgroup $H\leq\Gamma$, set $M_H=H\backslash\mathbb{H}^3$, and let $\pi:M_H\rightarrow M$ denote the natural covering map. If $\mathcal{N}_L(\TCC(M_H))$ isometrically embeds into $M$ under $\pi$, then $H$ is a free factor of $\Gamma$.
\end{prop}

Note that this conclusion does not hold if one uses $\CC(M_H)$ rather than $\TCC(M_H)$. Indeed, if $H$ is a cyclic group generated by a loxodromic isometry $\psi$, then $\CC(M_H)$ consists of a single closed geodesic. If $\ell(\psi)$ is sufficiently small, then a high radius metric neighborhood of $\CC(M_H)$ always includes into $M$ as a Margulis tube, though $\Gamma$ may not split as a non-trivial free product. 

\begin{proof}[Proof of Proposition \ref{prop-CCEmbedFactor}]
    It suffices to assume that $\Gamma$ is not cyclic or conjugate to a subgroup of $\Isom^+(\BH^2)$, the latter case following from standard area bounds. Let $\eta>0$ be the constant given by Bowditch in \cite{BOW} so that $\CC(M)\subset M_{\leq\eta}\cup W\cup\CN_\eta(C)$. As in Lemma \ref{Lemma-TCCNeighborhood}, fix $D=D(\eta+1)>0$ so that $(M_H)_{\leq\eta +1}\subset\mathcal{N}_D(\TCC(M_H))$. Define \[L=D+(2\eta+\xi)\cdot2\rank(M)+\eta+2\] and suppose that $\mathcal{N}_L(\TCC(M_H))$ embeds isometrically into $M$ under $\pi$.\par

    Set \[N=\{p\in M_H \;|\; D\leq d(p, \TCC(M_H))\leq L-(\eta+1)\}\subset M_H\] and for $t\geq 0$, define \[S_t=\partial \CN_t(\TCC(M_H)).\] Note that $N$ separates $M_H$ and $\pi(N)$ separates $M$. By the choice of $D$, we have $N\subset(M_H)_{>\eta}$. For any $q\in N$, $B_{M_H}(q,\eta/2)$ therefore isometrically embeds into $M$ under $\pi$, so we have $\pi(N)\subset M_{>\eta}$. Thus, $\CC(M)$ may only intersect $\pi(N)$ in $W\cup\CN_\eta(C)$, by the choice of $\eta$. \par

    \begin{figure}
        \centering
        \includegraphics[width=0.5\linewidth]{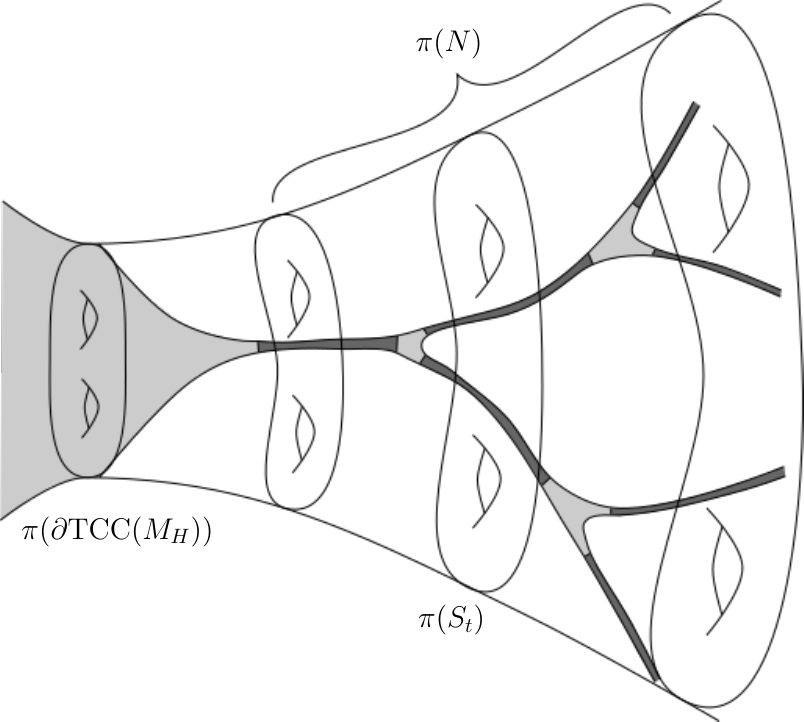}
        \caption{A sketch of the proof of Proposition \ref{prop-CCEmbedFactor}. $W$ is shown in dark gray and $\CC(M)- W$ is shown in light gray.}
        \label{fig:FreeFactor}
    \end{figure}

    For each compressing disk $D\subset C$, $\diam(\CN_\eta(D))\leq2\eta+\xi$. Thus, since $C$ consists of at most $2\rank(M)$ such disks, it follows from the choice of $L$ and $N$ that there exists some $t\in[D,D+(2\eta+\xi)\cdot2\rank(M)+1]$ such that $S_t\subset N$ and \[\pi(S_t)\cap\CC(M)\subset W.\] A cartoon of this is shown in Figure \ref{fig:FreeFactor}. For each $\xi$-handle $h$ that intersects $\pi(S_t)$, choose one compressing disk $D_h\subset h$ that has diameter at most $\xi$, intersects $\pi(S_t)$, and is properly embedded in $\CC(M)$. Define $\mathcal{C}\subset W$ to be the disjoint union of these disks $D_h$.\par

    Let $X$ be the component of $\CC(M)- \mathcal{C}$ containing $\pi(\CC(M_H))$. Fix $p'\in \CC(M_H)$ and let $p=\pi(p')$. Observe that $\pi_1(X,p)$ is a free factor of $\pi_1(\CC(M),p)$, since $X$ is a component of the result of removing a disjoint union of disks from $\CC(M)$. Additionally, we have \[X\subset \pi(\mathcal{N}_L(\TCC(M_H))),\] since $\pi(S_t)$ separates $\pi(\CC(M_H))$ from $\partial\CN_L(\TCC(M_H))$, $\pi(S_t)\cap\CC(M)\subset W$, and $\mathcal{C}$ separates the $\xi$-handles that intersect $\pi(S_t)$.\par

    Where $p'\in \CC(M_H)$ is as above, any loop $\beta':[0,1]\rightarrow M_H$ with $\beta'(0)=\beta'(1)=p'$ is pointed homotopic into $\CC(M_H)$. Thus, since $X\subset \pi(\mathcal{N}_L(\TCC(M_H)))$, any loop $\beta:[0,1]\rightarrow X$ with $\beta(0)=\beta(1)=p$ is pointed homotopic in $M$ into $\pi(\CC(M_H))$. Therefore, we see that $\pi_1(X,p)=\pi_1(\pi(\CC(M_H)),p)$. Let $\pi_*:\pi_1(M_H,p')\rightarrow\pi_1(M,p)$ be the injective homomorphism induced from the projection map $\pi$, and note that \[\pi_*(\pi_1(M_H,p'))=\pi_1(\pi(\CC(M_H)),p)=\pi_1(X,p).\] Thus, we have shown that $\pi_*(\pi_1(M_H,p'))$ is a free factor of $\pi_1(\CC(M),p)=\pi_1(M,p)$. The subgroup $\pi_*(\pi_1(M_H,p'))$ maps to $H$ under an appropriate holonomy representation $\pi_1(M,p)\longrightarrow \Gamma$, confirming that $H$ is a free factor of $\Gamma$.
\end{proof}

The general strategy for proving the \hyperlink{thm:decomp}{Chromatography Theorem} will be to use Proposition \ref{prop-CCEmbedFactor} to justify that $J_{1,t}^*(P(1))$ is a free factor of $P(t)$, for a path $P:[0,1]\longrightarrow\CD_3$ as in the theorem. To apply Proposition \ref{prop-CCEmbedFactor}, we will first need to be able to control $\TCH(J_{1,t}(P(1)))$ for $t$ close to $1$. The next two lemmas will address this.\par

For the next lemma, it will be useful to have the language of Chabauty convergence of closed sets in $\mathbb{H}^3$. Letting $\mathcal{C}(\mathbb{H}^3)$ denote the set of closed subsets of $\mathbb{H}^3$, the conditions for convergence in the Chabauty topology on $\mathcal{C}(\mathbb{H}^3)$ directly parallel those for convergence in the Chabauty topology on $\CD_n$, as described in Definition \ref{def:Chabauty}. More precisely, we have $X_n\rightarrow X$ in $\mathcal{C}(\mathbb{H}^3)$ if (i) every accumulation point of a sequence $x_n\in X_n$ is contained in $X$, and (ii) every $x\in X$ is the limit of a sequence $x_n\in X_n$. See \cite[Chapter E]{BEN} for more information on the Chabauty topology.\par

For a non-trivial convex cocompact group $\Gamma\in\CD_3$, we also let $\mathcal{P}(\Gamma)$ denote the \textit{Dirichlet fundamental polyhedron} for $\Gamma$ based at a fixed basepoint $O\in\mathbb{H}^3$. That is, \[\CP(\Gamma):=\{p\in\BH^3\;|\;d(p, O)\leq d(\psi(p),O)\;\forall\psi\in\Gamma\}.\] We also define \[\sys(\Gamma):=\min_{\psi\in\Gamma-\{\id\}}\ell(\psi),\] which is frequently called the \textit{systole} of $\Gamma\backslash \mathbb{H}^3$. \par

For a group $\Gamma\in\CD_3$, we say that a sequence of discrete injective representations $\rho_n:\Gamma\longrightarrow\Isom^+(\BH^3)$ \textit{converges strongly} to a discrete injective representation $\rho:\Gamma\longrightarrow\Isom^+(\BH^3)$ if $\rho_n(\psi)\rightarrow\rho(\psi)$ for all $\psi\in\Gamma$, and $\rho_n(\Gamma)\rightarrow\rho(\Gamma)$ in $\CD_3$. A discrete injective representation $\rho:\Gamma\longrightarrow\Isom^+(\BH^3)$ is said to be \textit{convex cocompact} if its image is convex cocompact.

\begin{lemma}
    \label{lemma-CCSeqFacts}
    For any convex cocompact representation $\rho:\Gamma\longrightarrow\Isom^+(\BH^3)$ and $D\geq0$, there exists a compact set $K\subset\BH^3$ such that for any sequence $\rho_n:\Gamma\longrightarrow\Isom^+(\BH^3)$ of convex cocompact representations converging strongly to $\rho$, we have $\mathcal{N}_D(\TCH(\rho_n(\Gamma)))\subset\rho_n(\Gamma)\cdot K$ for $n$ sufficiently large.
\end{lemma}

\begin{proof} It suffices to assume that $\Gamma$ is non-trivial, otherwise each $\TCH(\rho_n(\Gamma))=\varnothing$.

    \textit{Step 1:} We first construct a compact set $K'\subset\BH^3$ depending just on $\rho$ such that \[\hull(\rho_n(\Gamma))\subset \rho_n(\Gamma)\cdot K'\] for $n$ sufficiently large. Define \[A_n=\mathcal{P}(\rho_n(\Gamma))\cap \hull(\rho_n(\Gamma)),\;\;\;\;\;\; A=\mathcal{P}(\rho(\Gamma))\cap \hull(\rho(\Gamma)).\] Fixing $\varepsilon>0$ and letting $K'=\mathcal{N}_\varepsilon(A)$, it will suffice to show that $A_n\subset K'$ for $n$ sufficiently large.\par

    A result of J{\o}rgensen and Marden \cite{JORMAR} implies that the limit sets $\Lambda(\rho_n(\Gamma))$ converge to $\Lambda(\rho(\Gamma))$ in the Hausdorff topology on $\partial\overline{\BH^3}$. This, combined with a result of Bowditch on convex hulls \cite{BOWCHS}, implies that $\hull(\rho_n(\Gamma))\rightarrow \hull(\rho(\Gamma))$ in the Chabauty topology on $\mathcal{C}(\mathbb{H}^3)$. J{\o}rgensen and Marden \cite{JORMAR} also establish that strong convergence implies convergence of the Dirichlet fundamental polyhedra. More precisely, it follows that for all $R>0$ \[\CN_R(\{O\})\cap \mathcal{P}(\rho_n(\Gamma))\rightarrow \CN_R(\{O\})\cap \mathcal{P}(\rho(\Gamma))\] in the Hausdorff topology on compact sets of $\mathbb{H}^3$, as $n\rightarrow\infty$. This implies that we also have $\mathcal{P}(\rho_n(\Gamma))\rightarrow\mathcal{P}(\rho(\Gamma))$ in the Chabauty topology on $\mathcal{C}(\mathbb{H}^3)$.\par

    We will show next that there exists a sequence of points $p_n\in A_n$ converging in $\mathbb{H}^3$ to a point $p\in A$. Choose $p\in A$ so that $p$ is in the interior of $\mathcal{P}(\rho(\Gamma))$. By the Chabauty convergence $\hull(\rho_n(\Gamma))\rightarrow \hull(\rho(\Gamma))$, there exists a sequence $p_n\in \hull(\rho_n(\Gamma))$ with $p_n\rightarrow p$. We see from the convergence of Dirichlet polyhedra discussed in the preceding paragraph that $p_n$ must be contained in the interior of $\mathcal{P}(\rho_n(\Gamma))$ for $n$ sufficiently large, so $p_n\in A_n$ for $n$ sufficiently large.\par

    Now, suppose for contradiction that after passing to a subsequence, there exists a sequence $x_n\in A_n$ with $d(x_n,A)>\varepsilon$ for all $n$. By potentially replacing each $x_n$ with a point on the geodesic segment between $p_n$ and $x_n$, the convexity of $A_n$ and the compactness of $A$ allow us to assume that the sequence $\{x_n\}$ is bounded. Hence, the sequence $\{x_n\}$ has an accumulation point $x\in\mathbb{H}^3$, so the Chabauty convergence of the sequences $\hull(\rho_n(\Gamma))\rightarrow \hull(\rho(\Gamma))$ and $\mathcal{P}(\rho_n(\Gamma))\rightarrow\mathcal{P}(\rho(\Gamma))$ implies that $x\in A$. It must also be the case that $d(x,A)\geq \varepsilon$, so we have a contradiction.\par

    \textit{Step 2:} We show next that $\sys(\rho_n(\Gamma))\rightarrow\sys(\rho(\Gamma))$ as $n\rightarrow\infty$. First, pick $\psi\in\Gamma$ such that $\ell(\rho(\psi))=\sys(\rho(\Gamma))$. Since $\rho_n(\psi)\rightarrow\rho(\psi)$ in $\Isom^+(\BH^3)$, we have $\ell(\rho_n(\psi))\rightarrow\ell(\rho(\psi))$ as $n\rightarrow\infty$, so
    \begin{equation}
    \label{eqn-limsup}
        \limsup_n\sys(\rho_n(\Gamma))\leq\sys(\rho(\Gamma)).
    \end{equation}
    \par

    By the proof of step 1, there exists a compact set $K'\subset\mathbb{H}^3$ so that $K'$ contains a fundamental domain for the action of $\rho_n(\Gamma)$ on $\hull(\rho_n(\Gamma))$ for sufficiently large $n$. For such $n$, there is some $\psi_n\in\rho_n(\Gamma)$ so that $\ell(\psi_n)=\sys(\rho_n(\Gamma))$ and the axis $A_{\psi_n}$ intersects $K'$. Let $K''=\mathcal{N}_{\sys(\rho(\Gamma))}(K')$. The preceding paragraph tells us that $\sys(\rho_n(\Gamma))<2\sys(\rho(\Gamma))$ for $n$ sufficiently large; hence, for such $n$, $K''$ contains a geodesic segment $\alpha_n\subset A_{\psi_n}$ so that $A_{\psi_n}=\langle\psi_n\rangle\cdot\alpha_n$. In particular, we see $\psi_n\cdot K''\cap K''\neq\varnothing$ for $n$ sufficiently large, so it follows that the sequence $\{\psi_n\}_n\subset\Isom^+(\BH^3)$ is bounded.\par

     Now, choose a subsequence $\{n_i\}_i$ so that \[\liminf_n\sys(\rho_n(\Gamma))=\lim_{i\rightarrow\infty}\sys(\rho_{n_i}(\Gamma))=\lim_{i\rightarrow\infty}\ell(\psi_{n_i}).\] Since the sequence $\{\psi_{n_i}\}_i\subset\Isom^+(\BH^3)$ is bounded, there is a further subsequence, which we denote the same way, converging to some $\psi\in\Isom^+(\BH^3)$. Since $\rho_n(\Gamma)\rightarrow\rho(\Gamma)$ in $\CD_3$, the first condition for Chabauty convergence in $\CD_3$ in Definition \ref{def:Chabauty} implies that we must have $\psi\in\rho(\Gamma)$. Since each $\psi_{n_i}\neq\id$, it follows from Lemma \ref{Unbhd} that also $\psi\neq\{\id\}$. We then conclude that \begin{equation}
         \label{eqn-liminf}
         \sys(\rho(\Gamma))\leq\ell(\psi)=\lim_{i\rightarrow\infty}\ell(\psi_{n_i})=\liminf_n\sys(\rho_n(\Gamma)).
     \end{equation}
     Combining Equations \ref{eqn-limsup} and \ref{eqn-liminf} completes step 2.\par

     \textit{Step 3:} Next, we show that there exists $R=R(\rho)$ such that $\BH^3_{\leq 1}(\rho_n(\Gamma))\subset \mathcal{N}_R(\hull(\rho_n(\Gamma)))$ for $n$ sufficiently large. Fix $\delta$ with $0<\delta<\sys(\rho(\Gamma))$, and note that step 2 implies that $\sys(\rho_n(\Gamma))\geq\delta$ for $n$ sufficiently large. Let $R=R(\delta)$ be the constant given by Lemma \ref{lemma-LoxodromicFacts}(1) with respect to translation length $\delta$; for large $n$ and any $\psi\in\rho_n(\Gamma)-\{\id\}$, we then have $\ell(\psi)\geq\delta$, so the choice of $R$ implies that $T_1(\psi)\subset\mathcal{N}_R(A_\psi)$. Since $A_\psi\subset \hull(\rho_n(\Gamma))$, we have confirmed that $\BH^3_{\leq 1}(\rho_n(\Gamma))\subset\mathcal{N}_R(\hull(\rho_n(\Gamma)))$.\par
     
     \textit{Step 4}: We now finish the proof. Fix $D\geq 0$ and define $K=\mathcal{N}_{R+D}(K')$, where $K'\subset\mathbb{H}^3$ and $R>0$ are as in steps 1 and 3, respectively. Then, we see \[\mathcal{N}_D(\TCH(\rho_n(\Gamma)))\subset \mathcal{N}_{R+D}(\hull(\rho_n(\Gamma)))\subset \rho_n(\Gamma)\cdot\mathcal{N}_{R+D}(K')\] for $n$ sufficiently large, so the result follows.
\end{proof}

Before continuing to the next lemma, we recall that for any path $P:[0,1]\longrightarrow\CD_3$ and $s,t\in [0,1]$, Subsection \ref{subsec:path machinery} described a map $J_{s,t}:\Gamma_s\longrightarrow\overline{\Isom^+(\BH^3)}$ with image in $\Gamma_t\cup\{\infty\}$, where each $\Gamma_{t'}:=P(t')$. Maps $J_{s,t}^*:\Gamma_s\longrightarrow\Isom^+(\BH^3)$ were then defined with $J_{s,t}^*(\Gamma_s)=J_{s,t}(\Gamma_s)\cap\Isom^+(\BH^3)$, and Lemma \ref{Lemma-SubgroupConvg} says that for any subgroup $H\leq\Gamma_s$ and any $t\in[0,1]$, $J_{s,t}^*(H)$ is a subgroup of $\Gamma_t$.

\begin{lemma}
\label{lemma-PathTCCEmbed}
    Suppose $P:[0,1]\longrightarrow\CD_3$ is a path in the Chabauty topology such that $P(1)$ is convex cocompact. For each $t$, set \[M_t=P(t)\backslash\mathbb{H}^3,\hspace{.3in}M_t^J=(J_{1,t}^*(P(1)))\backslash\mathbb{H}^3,\] and let $\pi_t:M_t^J\rightarrow M_t$ be the induced covering map. Then, for any $D>0$, $\mathcal{N}_D(\TCC(M_t^J))$ isometrically embeds into $M_t$ under $\pi_t$, for $t$ sufficiently close to $1$.
\end{lemma}

\begin{proof}
    We let $\Gamma_t$ denote $P(t)$ for all $t$. By Lemma \ref{lem:injhomo}, $J_{1,t}$ is an injective homomorphism into $\Isom^+(\BH^3)$ for $t$ sufficiently close to $1$, in which case $J_{1,t}^*(\Gamma_1)=J_{1,t}(\Gamma_1)$. Additionally, Marden's stability theorem \cite{MAR} implies that $J_{1,t}(\Gamma_1)$ is convex cocompact for $t$ sufficiently close to $1$, and hence for such $t$ we have that $J_{1,t}:\Gamma_1\longrightarrow\Isom^+(\BH^3)$ is a convex cocompact representation. These representations strongly converge to the identity map, by Proposition \ref{prop:J Facts}(3) and the Chabauty topology continuity given by Lemma \ref{Lemma-SubgroupConvg}. Lemma \ref{lemma-CCSeqFacts} then implies that there exists a compact set $K\subset\mathbb{H}^3$ and $t_0\in[0,1)$ so that $\mathcal{N}_D(\TCH(J_{1,t}(\Gamma_1)))\subset J_{1,t}(\Gamma_1)\cdot K$ for all $t\geq t_0$.\par

    Now, let $\{t_n\}\subset[t_0,1)$ be a sequence with $t_n\rightarrow 1$ and suppose that for each $n$, $\psi_n\in\Gamma_{t_n}$ satisfies $\psi_n\cdot K\cap K\neq\varnothing$. Observe that the sequence $\{\psi_n\}\subset\Isom^+(\BH^3)$ must be bounded, since $K$ is compact. Passing to a subsequence, there then exists $\psi\in\Isom^+(\BH^3)$ such that $\psi_n\rightarrow\psi$, and the first condition in Definition \ref{def:Chabauty} for convergence in the Chabauty topology tells us that $\psi\in\Gamma_1$.\par

    By Lemma \ref{Unbhd}, there exists a neighborhood $U_{\psi}\subset\Isom^+(\BH^3)$ and $t_0'\geq t_0$ such that $|\Gamma_t\cap U_{\psi}|=1$ for all $t\geq t_0'$. Recall that for $t\geq t_0'$, $J_{1,t}(\psi)$ is defined to be the unique element in $\Gamma_t\cap U_\psi$. Since $\psi_n\rightarrow\psi$, we must have $\psi_n\in U_\psi$ and therefore $\psi_n=J_{1,t_n}(\psi)$, for $n$ sufficiently large.\par

    Thus, we conclude that for $t$ sufficiently close to $1$, if $\varphi\in\Gamma_t$ satisfies $\varphi\cdot K\cap K\neq\varnothing$, then we must have $\varphi\in J_{1,t}(\Gamma_1)$. It follows that for such $t$, if $\varphi'\in\Gamma_t$ satisfies
    \[\varphi'\cdot\mathcal{N}_D(\TCH(J_{1,t}(\Gamma_1)))\cap \mathcal{N}_D(\TCH(J_{1,t}(\Gamma_1)))\neq\varnothing,\] then $\varphi'\in J_{1,t}(\Gamma_1)$, which implies the claim.
\end{proof}

The following lemma will examine paths of convex cocompact Kleinian groups along which the isomorphism type of the group only changes at a single time. More precisely, we will consider a path $P:[0,1]\longrightarrow\CD_3$ for which $J_{0,t}$ maps $\Gamma_0:=P(0)$ isomorphically onto $\Gamma_t:=P(t)$ for all $t<1$, but perhaps some paths $t\mapsto J_{0,t}(\psi)$ diverge to $\infty$ as $t\rightarrow1$. We then ``pull back" the limit group $\Gamma_1$ to each $\Gamma_t$ using the map $J_{1,t}$, and examine how (the isomorphic image of) $\Gamma_1$ sits inside these approximating groups.\par

Recall that a path $P:I\longrightarrow\CD_3$ is said to be \textit{convex cocompact} if $P(t)$ is convex cocompact for every $t\in I$.

\begin{lemma}
    \label{Lemma-PathFreeFactorIso}
    If $P:[0,1]\longrightarrow\CD_3$ is a convex cocompact path such that $J_{0,t}(P(0))=P(t)$ for all $t<1$, then $J_{1,0}(P(1))$ is a free factor of $P(0)$.
\end{lemma}

\begin{proof}
    Let $\Gamma_t$ denote $P(t)$ for all $t$. Since $J_{0,t}(\Gamma_0)=\Gamma_t$ for all $t<1$, $J_{0,t}$ must be finite valued for all $t<1$. Hence, since each $\Gamma_t$ is convex cocompact, it follows from Lemma \ref{lem:injhomo} that for $t\in[0,1)$, the map $t\mapsto J_{0,t}$ is a path of injective convex cocompact representations of $\Gamma_0$ into $\Isom^+(\BH^3)$, with the continuity of this path of representations following from Proposition \ref{prop:J Facts}(3). Setting \[M_t:=\Gamma_t\backslash\BH^3\] for all $t$, it therefore follows from Marden's stability theorem \cite{MAR} that $M_t$ is homeomorphic to $M_0$ for all $t<1$.\par

    Let $D>0$ be the constant given by Proposition \ref{prop-CCEmbedFactor} with respect to the topology of $M_0$. By Lemmas \ref{lem:injhomo} and \ref{lemma-PathTCCEmbed}, there exists $s\in[0,1)$ close to $1$ such that $J_{1,s}$ is finite valued and $\CN_D(\TCC(M_s^J))$ isometrically embeds into $M_s$ under the induced covering map $M_s^J\longrightarrow M_s$, where \[M_s^J:=J_{1,s}(\Gamma_1)\backslash\BH^3.\]  Proposition \ref{prop-CCEmbedFactor} then implies that $J_{1,s}(\Gamma_1)$ is a free factor of $\Gamma_s$. Since $J_{0,s}(\Gamma_0)=\Gamma_s$, Proposition \ref{prop:J Facts}(5) implies that \begin{equation}
    \label{eqn: comp equality}
        J_{1,0}(\Gamma_1)=J_{s,0}(J_{1,s}(\Gamma_1)).
    \end{equation} It also follows that $J_{s,0}$ is an isomorphism from $\Gamma_s$ onto $\Gamma_0$, so since $J_{1,s}(\Gamma_1)$ is a free factor of $\Gamma_s$,  Equation \ref{eqn: comp equality} implies that $J_{1,0}(\Gamma_1)$ is a free factor of $\Gamma_0$.
\end{proof}

The remaining proofs of this section, including the proof of the \hyperlink{thm:decomp}{Chromatography Theorem}, will inductively make use of Lemma \ref{Lemma-PathFreeFactorIso}. In each case, the quantity that we will induct on will be the following slight modification of the \textit{rank} of a group. 

\begin{defn}[$\rank_0$]
\label{def: rank}
    Let $G$ be a finitely generated group. If $G$ is non-trivial, we define $\rank_0(G)$ to be the minimal size of a generating set for $G$. If $G$ is trivial, then $\rank_0(G)=0$.
\end{defn}

That is, $\rank_0$ agrees with the standard notion of rank for non-trivial groups, and we declare that $\rank_0(\{\id\})=0$.

\begin{lemma}
\label{lem: J is CC}
    If $P:[0,1]\longrightarrow\CD_3$ is a convex cocompact path such that $P(t)$ has infinite covolume for all $t$, then for every finitely generated subgroup $H\leq P(0)$, $J_{0,t}^*(H)$ is convex cocompact for all $t$.
\end{lemma}

Here, a group $\Gamma\in\CD_3$ is said to have infinite covolume if $\vol(\Gamma\backslash\BH^3)=\infty$.

\begin{proof}[Proof of Lemma \ref{lem: J is CC}]
    Denote $\Gamma_t:=P(t)$ for all $t$. We will induct on $\rank_0(H)$.\par
    \textit{Base case:} Suppose $\rank_0(H)=0$, in which case $H=\{\id\}$. Since $H$ contains a single element and $J_{0,t}^*(H)$ is a subgroup of $\Gamma_t$ for all $t$ by Lemma \ref{Lemma-SubgroupConvg}, we must have that $J_{0,t}^*(H)=\{\id\}$ for all $t$. The trivial group is convex cocompact, completing the base case.\par

    \textit{Inductive step:} Suppose $\rank_0(H)=n\geq 1$ and that the conclusion of the lemma holds for all subgroups $H'\leq \Gamma_0$ with $\rank_0(H')<n$.\par
    
    If the restriction $J_{0,t}|_H$ is finite valued for all $t$, then Lemma \ref{lem:injhomo} says that $J_{0,t}|_H$ is an injective homomorphism for all $t$; in this case, each $J_{0,t}^*(H)=J_{0,t}(H)$ is a finitely generated subgroup of the infinite covolume convex cocompact group $\Gamma_t$, so a covering theorem due to Thurston (see also \cite{DICKCOVER}) says that $J_{0,t}^*(H)$ is convex cocompact.\par

    It therefore suffices to assume that $J_{0,t}|_H$ is not finite valued for all $t$, in which case Proposition \ref{prop:J Facts} and Lemma \ref{lem:injhomo} imply that there exists some $s\in(0,1]$ such that $J_{0,t}|_H$ is finite valued if and only if $t\in[0,s)$. 

    \begin{claim*}
        $J_{0,s}^*(H)$ is finitely generated.
    \end{claim*}

    Before proving the claim, we will use it to complete the inductive step of the proof. Define the path $D:[0,s]\longrightarrow\CD_3$ by \[D(t)=J_{0,t}^*(H),\] with the continuity of $D$ given by Lemma \ref{Lemma-SubgroupConvg}. Since $J_{0,t}|_H$ is finite valued for $t\in[0,s)$, $J_{0,t}^*|_H$ is an injective homomorphism for such $t$, and so $D(t)$ is finitely generated for all $t\in[0,s]$, using the claim for $t=s$. Since each $D(t)$ is a finitely generated subgroup of the convex cocompact infinite covolume group $\Gamma_t$, the covering theorem again says each $D(t)$ is convex cocompact. It follows that (a reparameterization of) the path $D$ satisfies the hypotheses of Lemma \ref{Lemma-PathFreeFactorIso}, which allows us to conclude that \[H':=J_{s,0}(D(s))=J_{s,0}(J_{0,s}^*(H))\] is a free factor of $D(0)=H$. Since $J_{0,s}|_H$ is not finite valued, $H'$ must be a proper subgroup of $H$, so Grushko's theorem (see \cite{STA}) implies that $\rank_0(H')<\rank_0(H)$. The induction hypothesis therefore implies that $J_{0,t}^*(H')$ is convex cocompact for all $t\in[0,1]$. It follows from Proposition \ref{prop:J Facts}(5) that $H'=\{\psi\in\Gamma_0\;|\; s\in I_\psi^0\},$ so we have \[J_{0,t}^*(H)=
    \begin{cases}
        D(t), &t\in[0,s]\\
        J_{0,t}^*(H'), &t\in[s,1],
    \end{cases}\]
    confirming that $J_{0,t}^*(H)$ is convex cocompact for all $t$.\par

    \textit{We now proceed with the proof of the claim}. Since $\Gamma_s$ is finitely generated, Lemma \ref{lem:injhomo} implies that there exists some $t<s$ such that $J_{s,t}$ is finite valued. Lemma \ref{lem:injhomo} then also says that both $J_{0,t}|_H$ and $J_{s,t}$ are injective homomorphisms. \par

    Observe that \begin{equation}
        \label{eqn: intersection}
        J_{0,t}(H)\cap J_{s,t}(\Gamma_s)=J_{s,t}(J_{0,s}^*(H)).
    \end{equation}
    Indeed, if we fix $\psi\in J_{0,t}(H)\cap J_{s,t}(\Gamma_s)$, we have that $0,s\in I_\psi^t$, so Proposition \ref{prop:J Facts}(5) says that $J_{0,s}(J_{t,0}(\psi))=J_{t,s}(\psi)$. It follows that $J_{t,s}(\psi)\in J_{0,s}^*(H)$, and hence $\psi\in J_{s,t}(J_{0,s}^*(H))$, again by Proposition \ref{prop:J Facts}(2\&5). On the other hand, if $\psi'\in J_{s,t}(J_{0,s}^*(H))$, then there exists some $\varphi_0\in H$ such that $s\in I_{\varphi_0}^0$ and \[\psi'=J_{s,t}(J_{0,s}(\varphi_0)).\] It then follows from Proposition \ref{prop:J Facts}(5) that $J_{0,t}(\varphi_0)=\psi'$, which suffices to establish Equation \ref{eqn: intersection}.\par

    Finally, each of $J_{0,t}(H)$ and $J_{s,t}(\Gamma_s)$ are finitely generated, and hence convex cocompact by Thurston's covering theorem. By Equation \ref{eqn: intersection}, $J_{s,t}(J_{0,s}^*(H))$ is therefore an intersection of convex cocompact groups, so a result of Susskind \cite[Thm. 4]{SUSS} (see also \cite{Anderson-Intersections}) says that $J_{s,t}(J_{0,s}^*(H))$ is also convex cocompact. The injectivity of $J_{s,t}$ then confirms that $J_{0,s}^*(H)$ is finitely generated, completing the proof of the claim and the lemma.
\end{proof}

We make a few remarks about the proof of Lemma \ref{lem: J is CC}. In the special case that $H=\Gamma_0$ and $J_{0,t}(\Gamma_0)=\Gamma_t$ for all $t<1$, the work is reduced to showing that $J_{0,1}^*(\Gamma_0)$ is finitely generated. The proof presented here accomplishes this by using that $\Gamma_1$ is finitely generated, and citing a result of Susskind \cite{SUSS} about intersections of convex cocompact groups. A more direct proof would follow from a positive answer to the following question.

\begin{question}
    If $P:[0,1]\longrightarrow\CD_3$ is a path such that $P(t)$ is finitely generated and isomorphic to $P(0)$ for all $t<1$, must $P(1)$ be finitely generated?
\end{question}

\noindent Note that there are classical examples constructed by Thurston \cite[Sec. 7]{THUII} of \textit{sequences} of isomorphic finitely generated groups in $\CD_3$ that converge in the Chabauty topology to an infinitely generated group. As we have seen, paths in $\CD_3$ are more restricted than sequences, however, so it is not clear to the author whether similar phenomena can be observed along paths.\par

Finally, we can complete the proof of the \hyperlink{thm:decomp}{Chromatography Theorem}, providing a partial decomposition of a convex cocompact path $P:[0,1]\longrightarrow\CD_3$. Informally, the idea for the proof will be to repeatedly apply Lemma \ref{Lemma-PathFreeFactorIso} to inductively ``peel off" free factors from $P(0)$. Each successive free factor that is peeled off corresponds to an isomorphism bump path in the decomposition.

\begin{proof}[Proof of the \hyperlink{thm:decomp}{Chromatography Theorem}] We will prove the following claim:
    \begin{claim*}
        If $P:[0,1]\longrightarrow\CD_3$ is a convex cocompact path, then the path $D:[0,1]\longrightarrow\CD_3$ defined by \[D(t):=J_{0,t}^*(P(0))\] freely decomposes into a finite set of isomorphism bump paths.
    \end{claim*}
    
    In this claim, recall that Lemma \ref{Lemma-SubgroupConvg} implies that $D$ is continuous. Additionally, we have $D(0)=P(0)$ and $D(t)\leq P(t)$ for all $t$. To confirm that this claim implies the theorem, it remains to fix some $s\in[0,1]$ and show that $D(s)$ is a free factor of $P(s)$. Let $\Gamma_t$ denote $P(t)$ for all $t\in[0,1]$. Applying the claim to a reparameterization of the restricted domain path $P|_{[0,s]}$ (with the reversed direction) implies that the path $D':[0,s]\longrightarrow\CD_3$ defined by \[D'(t):=J_{s,t}^*(\Gamma_s)\] freely decomposes into a finite set of isomorphism bump paths. Letting $B_1',...,B_k':[0,s]\longrightarrow\CD_3$ be the set of paths in this free decomposition of $D'$ for which $B_i'(0)\neq\{\id\}$, we have that \[J_{s,0}^*(\Gamma_s)=\bigast_{1\leq i\leq k}B_i'(0).\] Additionally, Proposition \ref{prop:J Facts}(5) implies that $J_{s,0}^*(\Gamma_s)$ is the subgroup of elements of $\Gamma_0$ that $J_{0,s}$ is finite valued on, so we see that $J_{0,s}^*(\Gamma_0)=J_{0,s}^*(J_{s,0}^*(\Gamma_s))$. Hence, it follows that \[D(s)=J_{0,s}^*(\Gamma_0)=J_{0,s}^*(J_{s,0}^*(\Gamma_s))=J_{0,s}^*(\bigast_{1\leq i\leq k}B_i'(0))=\bigast_{1\leq i\leq k}B_i'(s),\] where the last equality follows from Lemmas \ref{lem:injhomo} and \ref{lem:bump vs J}, confirming that $D(s)$ is a free factor of $\Gamma_s$.\par
    
    \textit{We now proceed with the proof of the claim.} Fix a convex cocompact path $P:[0,1]\longrightarrow\CD_3$ and denote $\Gamma_t:=P(t)$ for all $t$. First, note that it suffices to assume that $\Gamma_t$ has infinite covolume for all $t$. Indeed, \cite[Thm. A]{ZevenbergenPaths1} implies that if some $\Gamma_s$ has finite covolume, then $J_{s,t}$ must be an isomorphism onto its image for all $t$, and it follows that the path $D$ appearing in the claim is itself an isomorphism bump path.\par 
    
    We assume, then, that $P(t)$ has infinite covolume for all $t$, in which case the proof of the claim will proceed by induction on $\rank_0(\Gamma_0)$, where $\rank_0(\Gamma_0)$ is as defined in Definition \ref{def: rank}. \par

    \textit{Base case:} Suppose $\rank_0(\Gamma_0)=0$. We then have that $\Gamma_0=\{\id\}$, so Lemma \ref{lem:injhomo} implies that $J_{0,t}^*(\Gamma_0)=\{\id\}$ for all $t\in[0,1]$. The path $D$ is therefore the constant path at $\{\id\}$, which is itself an isomorphism bump path.\par

    \textit{Inductive step:} Suppose $\rank_0(\Gamma_0)=n\geq 1$ and that the claim holds for all convex cocompact paths $P':[0,1]\longrightarrow\CD_3$ with $\rank_0(P'(0))<n$. The goal will be to show that the path $D:[0,1]\longrightarrow\CD_3$ defined by $D(t):=J_{0,t}^*(\Gamma_0)$ freely decomposes into finitely many isomorphism bump paths.\par

    First, assume that $J_{0,t}$ is finite valued, and therefore injective (by Lemma \ref{lem:injhomo}), for all $t\in[0,1]$. Proposition \ref{prop:J Facts}(3) then says that the map $[0,1]\longrightarrow\Hom(\Gamma_0,\Isom^+(\BH^3))$ given by $t\mapsto J_{0,t}^*$ is continuous, so $D$ itself is an isomorphism bump path.\par

    On the other hand, assume that it is not the case that $J_{0,t}$ is finite valued for all $t\in[0,1]$, in which case Proposition \ref{prop:J Facts} and Lemma \ref{lem:injhomo} imply that there exists $s\in(0,1]$ such that $J_{0,t}$ is finite valued if and only if $t\in[0,s)$, as in the inductive step of Lemma \ref{lem: J is CC}. Lemma \ref{lem: J is CC} implies that $D(t)$ is convex cocompact for all $t$, so the restricted domain path $D|_{[0,s]}$ satisfies the hypotheses of Lemma \ref{Lemma-PathFreeFactorIso}. We then have that $\Gamma_0$ splits as a free product \[\Gamma_0=H_s*K_s,\] where $H_s:=J_{s,0}(D(s))$. Informally, $H_s\leq \Gamma_0$ is the subgroup that consists of the elements of $\Gamma_0$ that ``survive" into $\Gamma_s$ without diverging to $\infty$ as they are tracked by the maps $J_{0,t}$. In particular, $H_s$ must be a proper subgroup of $\Gamma_0$, since $J_{0,s}$ is not finite valued.\par
    
    On the other hand, $J_{0,s}^*(K_s)=\{\id\}$. Furthermore, since $J_{0,t}$ is finite valued for all $t<s$, Lemma \ref{lem:injhomo} says that each such $J_{0,t}$ is an injective homomorphism, and Proposition \ref{prop:J Facts}(3) implies that the map $[0,s)\longrightarrow\Hom(K_s,\Isom^+(\BH^3))$ given by $t\mapsto J_{0,t}|_{K_s}$ is continuous. Therefore, the map $B:[0,1]\longrightarrow\CD_3$ defined by $B(t):=J_{0,t}^*(K_s)$ is an isomorphism bump path, with the continuity of $B$ given by Lemma \ref{Lemma-SubgroupConvg}.\par

    Now, define the path $D':[0,1]\longrightarrow\CD_3$ by $D'(t):=J_{0,t}^*(H_s)$, and note that Lemma \ref{lem: J is CC} says that $D'$ is convex cocompact. Since $H_s$ is a free factor properly contained in $\Gamma_0$, Grushko's theorem \cite{STA} implies that $\rank_0(H_s)<n$. The induction hypothesis therefore implies that $D'$ freely decomposes into finitely many isomorphism bump paths. Hence, to complete the proof, it suffices to show that $D$ freely decomposes into $B$ and $D'$. Lemma \ref{lem:injhomo} implies that $J_{0,t}^*$ is an isomorphism onto its image for $t<s$, so we have \[D(t)=J_{0,t}^*(\Gamma_0)=J_{0,t}^*(H_s*K_s)=J_{0,t}^*(H_s)*J_{0,t}^*(K_s)=D'(t)*B(t)\] for $t<s$. For $t\geq s$, Proposition \ref{prop:J Facts} implies that $D(t)=D'(t)$. Since $B(t)=\{\id\}$ for $t\geq s$, we have confirmed that \[D(t)=D'(t)*B(t)\] for all $t\in[0,1]$, as was to be shown.
\end{proof}

\subsection{A decomposition counterexample}

\label{subsec:decomp counterex}

The main goal for this final subsection is to construct the following example, which says that not every convex cocompact path can be decomposed in the strong way described by Definition \ref{def:free decomp}.

\begin{example}[Isomorphism bump decomposition counterexample]
\label{ex:decomp non-ex}
    There exists a convex cocompact path $P:[0,1]\longrightarrow\CD_2$ that does not freely decompose into a collection of isomorphism bump paths. 
\end{example}

To prove this example, we will recall the notion of a classical Schottky subgroup of $\Isom^+(\BH^2)$, a specific case of the standard combination theorems. Suppose that \[A_1^-,A_1^+,...,A_k^-,A_k^+\subset\BH^2\] are disjoint hyperbolic half-spaces, and that for each $i$, the isometry $\psi_i\in\Isom^+(\BH^2)$ satisfies \[\psi_i(\BH^2-\mathrm{int}(A_i^-))=A_i^+.\] Then, we say that $\Gamma=\langle\psi_1,...,\psi_k\rangle$ is a \textit{classical Schottky group}, and that $\{\psi_1,...,\psi_k\}$ is a \textit{classical Schottky generating set for $\Gamma$}. Note that not every generating set for a classical Schottky group need be a classical Schottky generating set (cf. \cite{Button}). A standard ping-pong argument (as discussed in Section \ref{subsec:combo thm}) shows that every classical Schottky generating set is a free generating set.

\begin{lemma}
\label{lemma: free group bases}
    There exists a free group $\Gamma=\langle A,B,C\rangle\leq\Isom^+(\BH^2)$ such that both $\{A,B,C\}$ and $\{B,C,D\}$ are classical Schottky generating sets for $\Gamma$, where $D=ABC$.
\end{lemma}

\begin{proof}
    Let $S=\Gamma\backslash\BH^2$ be a convex cocompact hyperbolic surface homeomorphic to a four holed sphere. The convex core of $S$ has four closed geodesic boundary components, which we denote $\alpha,\beta,\gamma,\delta$. Fix a base point $p\in S$ and choose loops $A,B,C\in\pi_1(S,p)$ such that $A,B,$ and $C$ are freely homotopic to $\alpha,\beta$, and $\gamma$, respectively, and such that the product loop $D:=ABC$ is freely homotopic to $\delta$. The set $\{A,B,C\}$ generates $\pi_1(S,p)$, and standard techniques\footnote{In brief, one cuts $S$ along a disjoint union of bi-infinite geodesics $g_A,g_B,g_C$ that meet $\alpha\&\delta$, $\beta\&\delta$, and $\gamma\&\delta$ orthogonally, respectively. These cuts result in a fundamental polygon in $\BH^2$ whose complement is a union of hyperbolic half-spaces, with side pairings $A,B$ and $C$.} (see \cite{Button}) show that $\{A,B,C\}$ is a classical Schottky generating set for $\Gamma$, under the identification of $\pi_1(S,p)$ with $\Gamma$.  \par

    Similarly, the loops $D, C^{-1}, B^{-1}$ are freely homotopic to (the unoriented loops) $\delta,\gamma,\beta$, respectively, and satisfy $DC^{-1}B^{-1}=A$, so the same argument implies that $\{D, C^{-1}, B^{-1}\}$ is a classical Schottky generating set for $\Gamma$, and so $\{B,C,D\}$ is as well.
\end{proof}

We end by completing the proof of Example \ref{ex:decomp non-ex}, constructing an example of a path that does not freely decompose into a collection of isomorphism bump paths.

\begin{proof}[Proof of Example \ref{ex:decomp non-ex}]
    
    Let $\Gamma=\langle A,B,C\rangle\leq\Isom^+(\BH^2)$ be a convex cocompact group as given by Lemma \ref{lemma: free group bases} so that both $\langle A,B,C\rangle$ and $\langle B,C,D\rangle$ are classical Schottky generating sets for $\Gamma$, where $D=ABC$. Our goal will be to construct a convex cocompact path $P:[0,1]\longrightarrow\CD_2$ such that

    \begin{enumerate}
        \item $P(1/2)=\Gamma$, 
        \item $P(0)=\langle A,C\rangle$,
        \item $P(1)=\langle B,D\rangle$,
        \item $J_{0,t}^*(\psi)=\psi$ for all $\psi\in \langle A,C\rangle$ and $t\in[0,1/2]$, and
        \item $J_{1,t}^*(\psi)=\psi$ for all $\psi\in \langle B,D\rangle$ and $t\in[1/2,1]$.
    \end{enumerate}

    Before carrying out the construction, we suppose that $P:[0,1]\longrightarrow\CD_2$ is a path with these named properties, and will show that $P$ suffices to prove the example. Accordingly, suppose for contradiction that $P$ freely decomposes into a collection $\CP$ of isomorphism bump paths. Let $B_1,...,B_k:[0,1]\longrightarrow\CD_2$ be the paths in $\CP$ that do not map $0$ to the trivial group, and let $B_1',...,B_{k'}':[0,1]\longrightarrow\CD_2$ be the paths in $\CP$ that do not map $1$ to the trivial group. It then follows from property (4) of $P$ and Lemma \ref{lem:bump vs J} that each restriction $B_i|_{[0,1/2]}$ is constant, for $1\leq i\leq k$. Similarly, property (5) implies that each $B_i'|_{[1/2,1]}$ is constant for $1\leq i\leq k'$. We then have  \begin{equation}
    \label{eqn: free factors}
        \langle A,C\rangle=\langle B_i(1/2)\rangle_{1\leq i\leq k}\hspace{.3in}\text{ and }\hspace{.3in}\langle B,D\rangle=\langle B_i'(1/2)\rangle_{1\leq i\leq k'}.
    \end{equation} Since $\Gamma=\langle A,B,C\rangle$ is free and $D=ABC$, we see that $\langle A,C\rangle\cap\langle B,D\rangle=\{\id\}$. Thus, it must be the case that all the paths $B_i$ and $B_j'$ are distinct, so Equation \ref{eqn: free factors} and Definition \ref{def:free decomp} of a free decomposition implies that $\langle A,B ,C\rangle$ has a free splitting with factors $\langle A,C\rangle$ and $\langle B,D\rangle$. Grushko's theorem \cite{STA}, however, then implies that $\rank(\langle A,B,C\rangle)\geq 4$, a contradiction.\par

    We now conclude with the construction of the desired path $P$. Let $L_B>0$ be the translation length of the hyperbolic isometry $B\in\Isom^+(\BH^2)$. For $t\in(0,1/2]$, let $B_t\in \Isom^+(\BH^2)$ denote the hyperbolic isometry with the same attracting and repelling fixed points as $B$, such that $B_t$ has translation length $L_B/2t$. Note that $B_{1/2}=B$ and $L_B/2t\rightarrow\infty$ as $t\rightarrow 0$. Similarly, let $L_C$ denote the translation length of $C\in \Isom^+(\BH^2)$, and for $t\in[1/2,1)$, let $C_t\in\Isom^+(\BH^2)$ be the hyperbolic isometry with the same attracting and repelling fixed points as $C$, such that $C_t$ has translation length $L_C/2(1-t)$. We then have $C_{1/2}=C$ and $L_C/2(1-t)\rightarrow\infty$ as $t\rightarrow 1$.\par

    Define the map $P:[0,1]\longrightarrow\CD_2$ by \[P(t):=\begin{cases}
        \langle A,C\rangle, &t=0\\
        \langle A,B_t,C\rangle, &0<t\leq 1/2\\
        \langle B,C_t,D\rangle, &1/2\leq t<1\\
        \langle B,D\rangle, &t=1.
    \end{cases}\]
    The continuity of $P$ follows from the \hyperlink{thm:combo}{Combination Theorem}. Indeed, since $\langle A,B,C\rangle$ is a classical Schottky generating set for $\Gamma$, we can choose disjoint closed hyperbolic half-spaces $X_A^-,X_A^+,X_B^-,X_B^+, X_C^-,X_C^+\subset\BH^2$ such that \[\alpha(\BH^2-\mathrm{int}(X_\alpha^-))=X_\alpha^+,\] for each $\alpha\in\{A,B,C\}$. Set $X_\alpha=X_\alpha^-\cup X_\alpha^+$, for each $\alpha\in\{A,B,C\}$. Define paths $P_A,P_B,P_C:[0,1/2]\longrightarrow\CD_2$ so that $P_A$ and $P_C$ are the constant paths at $\langle A\rangle$ and $\langle C\rangle$, respectively, and \[P_B(t):=\begin{cases}
        \{\id\},&t=0\\
        \langle B_t\rangle, &0<t\leq1/2.
    \end{cases}\] The continuity of $P_B$ follows from the definition of the generators $B_t$. We then have that \[\psi_t(\BH^2-X_\alpha)\subset X_\alpha\] for each $\alpha\in\{A,B,C\}$, $t\in[0,1/2]$, and $\psi_t\in P_\alpha(t)$. As $P(t)=\langle P_A(t), P_B(t), P_C(t)\rangle$ for all $t\in[0,1/2]$, the continuity of $P$ on the interval $[0,1/2]$ indeed follows from the \hyperlink{thm:combo}{Combination Theorem}, and a symmetric argument justifies the continuity of $P$ on the interval $[1/2,1]$, using that $\{B,C,D\}$ is also a classical Schottky generating set for $\Gamma$. That $P$ satisfies the five desired properties follows immediately from the construction and the definition of the maps $J_{s,t}^*$.
\end{proof}

\bibliographystyle{plain}
\bibliography{bibliography}

\end{document}